\documentclass[11pt]{amsart}

\usepackage[T1]{fontenc}
\usepackage[utf8]{inputenc}
\usepackage{textcomp}
\usepackage{lmodern}
\usepackage{microtype}

\usepackage{amssymb}
\usepackage{amsthm}
\usepackage{amscd}

\usepackage{hyperref}

\usepackage{mathscinet}
\usepackage[giveninits=true,doi=false,url=false,sortcites,backend=biber,backref=true]{biblatex}
\newtheorem{theorem}{Theorem}[section]

\newtheorem{lemma}[theorem]{Lemma}

\newtheorem{cor}[theorem]{Corollary}

\theoremstyle{definition}
\newtheorem{definition}[theorem]{Definition}

\allowdisplaybreaks[2]

\newcommand{\ACFm}{\ensuremath{\mathrm{ACF}(\mathfrak{m},\Sigma)}}
\newcommand{\ACF}{\ensuremath{\mathrm{ACF}}}
\newcommand{\RCF}{\ensuremath{\mathrm{RCF}_0}}
\newcommand{\RCFU}{\ensuremath{\mathrm{RCF}(\Gamma,\Sigma)}}

\begin{document}

\title{From Saturated Embedding Tests to Explicit Algorithms}
\author{Henry Towsner}
\date{\today}
\thanks{Partially supported by NSF grant DMS-2054379}
\address {Department of Mathematics, University of Pennsylvania, 209 South 33rd Street, Philadelphia, PA 19104-6395, USA}
\email{htowsner@math.upenn.edu}
\urladdr{\url{http://www.math.upenn.edu/~htowsner}}

\begin{abstract}
  Quantifier elimination theorems show that each formula in a certain theory is equivalent to a formula of a specific form---usually a quantifier-free one, sometimes in an extended language. Model theoretic embedding tests are a frequently used tool for proving such results without providing an explicit algorithm.

  We explain how proof mining methods can be adapted to apply to embedding tests, and provide two explicit examples, giving algorithms for theories of algebraic and real closed fields with a distinguished small subgroup corresponding to the embedding test proofs given by van den Dries and G\"unaydin  \cite{MR2235481}.
\end{abstract}

\maketitle

\section{Introduction}

In this paper we consider a case where model theoretic arguments are used to indirectly proof the existence of an algorithm---embedding tests for proving quantifier elimination---and discuss, with some examples, how to extract an explicit algorithm from such a proof.

When $T$ is a computable (or computably enumerable) theory, the statement that $T$ admits quantifiers can itself be formulated as a $\Pi_2$ sentence in the language of arithmetic: elimination of quantifiers means that, for each formula $\phi(\vec x)$ there is a quantifier-free formula $\psi(\vec x)$ and a proof in first-order logic showing that $T\vdash\forall\vec x\,\phi(\vec x)\leftrightarrow\psi(\vec x)$.

When we have a true $\Pi_2$ sentence, we always have a corresponding total computable function: given $\phi(\vec x)$, we can simply search blindly through all possible quantifier-free formula $\psi(\vec x)$ and all possible proofs, knowing that we must eventually find the desired pair.

It is one of the central insights of proof theory that when we have a \emph{proof} that a $\Pi_2$ sentence is true, we should expect more: we should expect that a more perspicacious algorithm than a blind search is in fact hidden in our proof. Methods for actually extracting such algorithms, from actual proofs as written by mathematicians, is one of the central goals of proof mining \cite{kohlenbach:MR2445721}.

The formal methods of proof mining generally require that one begins with a proof of a $\Pi_2$ sentence in a reasonable theory, typically an extension of Peano arithmetic, perhaps with higher types. Embedding tests\footnote{There are many variations. See \cite{yin2007equivalence} for a comparison of several.} dealing with actual models---especially those, like the ones considered in this paper, which deal with uncountable \emph{saturated} models---fall outside the range of these methods. 

While it is likely possible to extend the methods of proof mining to cover the such models (functional interpretations, one of the main tools of proof mining, do exist for some constructive set theories, e.g. \cite{MR1778934}), in this paper we will instead show how we can reinterpret embedding tests in a computable way, and then will show how we can extract algorithms from this reinterpretation.

The author is grateful to Chris Miller for the suggestion that motivated this work.

\section{Embedding Tests and Proof Mining}

A prototypical example of an embedding test is the following. (This, or something close to it, can be found in any standard model theory textbook, e.g. \cite{Marker_2011}.)

\begin{theorem}\label{thm:qe}
  Let $T$ be a theory. Suppose that whenever $\mathcal{M}$ and $\mathcal{N}$ are models of $T$, $A\subseteq\mathcal{M}$, $f:A\rightarrow\mathcal{N}$ is an embedding, and $a\in|\mathcal{M}|\setminus A$, there is an $\mathcal{N}'\succ\mathcal{N}$ and an embedding $g:A\cup\{a\}\rightarrow\mathcal{N}'$ extending $f$.

  Then $T$ has quantifier elimination.
\end{theorem}
\begin{proof}
  As usual, to show quantifier elimination it suffices to show that whenever $\phi(x,\vec y)$ is a quantifier-free formula with the displayed free variables, there is a quantifier-free formula $\psi(\vec y)$ so that $T\vdash \psi(\vec y)\leftrightarrow\exists x\,\phi(x,\vec y)$.

  So let $\phi(x,\vec y)$ be given. We work in an extension of the language of $T$ with some fresh constant symbols and consider $\phi(x,\vec d)$ where $\vec d$ are the fresh constant symbols. Let $\Gamma$ be the quantifier-free consequences of $T\cup\{\exists x\,\phi(x,\vec d)\}$. 

  If $T\cup\Gamma\cup\{\forall x\,\neg\phi(x,\vec d)\}$ is inconsistent then by compactness there is some finite $\Gamma_0\subseteq\Gamma$ so that $T\cup\Gamma_0$ implies $\exists x\,\phi(x,\vec d)$, and therefore $T$ proves that $\bigwedge\Gamma_0$ is equivalent to $\exists x\,\phi(x,\vec d)$.\footnote{There is a minor technical point here: we need to be sure that $\bigwedge\Gamma_0$ is actually a formula. When $\Gamma_0$ is a finite \emph{non-empty} set, there is no issue. However if $|\vec y|=0$ and the language has no constant symbols then $\Gamma$ could actually be empty. Some presentations of these results add extra hypotheses---for instance, that there is at least one constant symbol---to ensure that this does not happen. We prefer the convention from proof theory that the $0$-ary connective $\bot$ should be present in the definition of first-order logic, and therefore $\bigwedge\emptyset$ is the formula $\neg\bot$.}

  Otherwise, $T\cup\Gamma\cup\{\forall x\,\neg\phi(x,\vec d)\}$ is consistent, and therefore has a model $\mathcal{N}$. Let $\Sigma$ be the set of quantifier-free sentences true in $\mathcal{N}$. If $T\cup\Sigma\cup\{\exists x\,\phi(x,\vec d)\}$ is inconsistent then $T\cup\Sigma\vdash\forall x\,\neg\phi(x,\vec d)$.  By compactness, there is a finite $\Sigma_0\subseteq\Sigma$ so that $T\cup\Sigma_0\vdash\forall x\,\neg\phi(x,\vec d)$, and therefore $T\vdash\bigwedge\Sigma_0\rightarrow\forall x\,\neg\phi(x,\vec d)$. But this means that $T\vdash\exists x\,\phi(x,\vec d)\rightarrow\bigvee\neg\Sigma_0$, and therefore $\bigvee\neg\Sigma_0\in \Gamma$, contradicting the construction of $\mathcal{N}$.

  Therefore $T\cup\Sigma\cup\{\exists x\,\phi(x,\vec d)\}$ must be consistent. Then it has a model, $\mathcal{M}$. Since $\mathcal{M}$ and $\mathcal{N}$ satisfy the same quantifier-free sentences about the constants $\vec d$, the map from $\vec d^{\mathcal{M}}$ to $\vec d^{\mathcal{N}}$ is an embedding.

  Pick $a\in|\mathcal{M}|$ so that $\mathcal{M}\vDash\phi(a,\vec d)$.   The assumption gives us a $\mathcal{N}'\succ\mathcal{N}$ and a homomorphism $g$ mapping $\vec d$ to itself and $a$ to some element $g(a)$ so that $\mathcal{N}'\vDash \phi(a,\vec d)$. Then $\mathcal{N}'\vDash\exists x\,\phi(a,\vec d)$, contradicting $\mathcal{N}\prec\mathcal{N}'$.
\end{proof}

When $T$ is a computable (or computably enumerable) theory in a countable language, the statement that $T$ has quantifier elimination can be encoded in arithmetic: given any (code for) an existential formula, there is a (code for) a quantifier-free formula together with a deduction from $T$ of the equivalence.

We would like to rephrase this theorem so that the statement and proof can be encoded in arithmetic as well, so that it will be amenable to proof mining approaches to extracting an algorithm from it.

The main idea is to replace models with algorithms that output formulas. In place of a countable model, we could hope to have an algorithm which outputs a description of that model---that is, outputs formulas (with parameters) true in the model. For instance, at stage $n$, our algorithm could output a partial descriptions of the model, and then at stage $n+1$, output a larger partial description of the model, so that eventually every every sentence or its negation appears.

Since many interesting models have theories which are not computable, it's not reasonable to expect that our algorithm \emph{only} outputs sentences true in the model. Instead, we have to allow our algorithms to change their mind.

\begin{definition}
  An \emph{approximation of a model} in the language $\mathcal{L}$ is a function $h$ such that:
  \begin{itemize}
  \item for each $n$, $h(n)$ is a finite set of formulas in $\mathcal{L}_{\mathbb{N}}$ (the language $\mathcal{L}$ with fresh constant symbols for each element of $\mathbb{N}$),
  \item for each formula $\phi$, $\lim_n\chi_{h(n)}(\phi)$ exists---that is, $\phi$ is either present in cofinitely many of the sets $h(n)$ or in finitely many of them,
  \item the set $\lim h$ of sentences $\phi$ such that $\lim_n\chi_{h(n)}\phi=1$ has the following properties:
    \begin{itemize}
    \item if $\sigma\in T$ then $\sigma\in \lim h$,
    \item $\lim f$ is closed under consequences: if $\lim h\vdash\phi$ then $\phi\in \lim f$,
    \item for every $\phi$, either $\phi\in \lim h$ or $\neg\phi\in \lim h$,
    \item if $\exists x\,\phi(n)\in\lim h$ then there is some constant $n$ so that $\phi(n)\in\lim h$.
    \end{itemize}
  \end{itemize}

  We say $h$ is consistent if $\bot\not\in\lim h$.
\end{definition}
This is almost the same as saying that there is a model $\mathcal{M}$ of $T$ whose universe is a quotient of $\mathbb{N}$ and so that $\lim f$ is the true sentences in $\mathcal{M}$, except that we do not require that $\lim f$ be consistent.

Note that we can easily construct such approximations in a fairly uniform way---for instance, given an index for a c.e. set $W_e$ of formulas from $\mathcal{L}$, we can describe an approximation of a model $h_e$ of $T\cup W_e$ which, further, is consistent so long as $T\cup W_e$ is consistent. At stage $n$ we consider the first $n$ formulas in some computable enumeration and output those formulas which are in $T$ or which have already been enumerated into $W_e$ by this stage. For the remaining formulas, we consider them in order, search for $n$ steps for a proof that the formula contradicts formulas we have already added, and add the formula precisely if we do not find a proof of a contradiction. (Actually, this method would work just as well if we replaced $W_e$ by a suitable coding of a $\Delta_2$ set, which will be useful below.)

To ensure the witnessing property, we adapt the usual trick from Henkin's proof of the completeness theorem: for each formula $\phi$ we pick a witness $w$ and add the formula $\exists x\,\phi(x)\rightarrow\phi(w)$, so that closure under consequence ensures that $w$ will be the needed witness. To make sure we have enough witnesses, we need to fix infinitely many pairwise disjoint subsets of the natural numbers---say, the numbers divisible by $2^k$ but not $2^{k+1}$ for various values of $k$---to serve as witnesses. Formally, for each formula $\phi$, let $k_\phi\geq 1$ be least so that no constant symbol appearing in $\phi$ is divisible by $2^k$. We choose a computable injective map $w$ from those $\phi$ with $k_\phi=k$ to those numbers divisible by $2^k$ but not $2^{k+1}$ and we require that $f$ make the formulas $\exists x\,\phi(x)\rightarrow \phi(w(\phi))$ true. 


We can directly adapt the statement and proof of Theorem \ref{thm:qe} to this setting. An embedding of approximations of models $h,h'$ is just a function $f:\mathbb{N}\rightarrow\mathbb{N}$ so that whenever $\phi$ is a quantifier-free sentence, $\phi\in\lim h$ iff $f(\phi)\in\lim h'$ (where $f(\phi)$ means the result of replacing constants from $\mathbb{N}$ in $\phi$ according to $f$), and an embedding is elementary if this condition holds for all sentences, not just quantifier-free ones.

\begin{theorem}
  Let $T$ be a computably enumerable theory. Suppose that whenever $h$ and $h'$ are consistent approximations of models of $T$, $A\subseteq\mathbb{N}$, and $f:A\rightarrow\mathbb{N}$ is an embedding of $h\upharpoonright A$ into $h'$, and $a\in \mathbb{N}\setminus A$, there is an approximation of a consistent approximation of a model $h''$, an elementary embedding $f'$ from $h'$ into $h''$, and an embedding $g$ from $A\cup\{a\}$ into $h''$ extending $f'\circ f$.

  Then $T$ has quantifier elimination.
\end{theorem}
\begin{proof}
  The proof is basically unchanged. Given a quantifier-free formula $\phi(x,\vec y)$, we pick some constants $\vec d$ from $\mathbb{N}$. Let $\Gamma$ be the quantifier-free consequences of $T\cup\{\exists x\,\phi(x,\vec d)\}$ and let $h'$ be an approximation of a model of $T\cup\Gamma\cup\{\forall x\,\neg\phi(x,\vec d)\}$ of the kind described above---that is, we pick a collection of Henkin sentences $\Psi$ and then at the $n$-th stage, $h'$ outputs those elements of $T\cup\Gamma\cup\{\forall x\,\neg\phi(x,\vec d)\}\cup\Psi$ which have been enumerated by stage $n$ and then resolves the remaining formulas in order.

  If we let $\Sigma$ be the set of quantifier-free sentences with constants from $\vec d$ true in $h'$, we can now try to construct an approximation $h$ of a model of $T\cup\{\phi(a,\vec d)\}\cup\Psi\cup\Sigma$ where $a$ is some new constant from $\mathbb{N}$. $\Sigma$ is not computably enumerable, but it is $\Delta_2$, and similar ideas work: at the $n$-th stage, $h$ outputs those sentences from $T\cup\{\phi(a,\vec d)\}\cup\Psi$ which have been enumerated by the $n$-th stage together with exactly what $h'$ thinks $\Sigma$ is at stage $n$---that is, every quantifier-free sentence with constants from $\vec d$ in $h'(n)$ is also in $h(n)$. Then we may take $f$ to be the identity on $\vec d$, since we have ensured that $h$ and $h'$ agree on quantifier-free sentences with constants from $\vec d$.

  By assumption, this means there is an extension $h''$ of $h'$. But $h''$ cannot be consistent, since it has to contain both $\forall x\,\neg \phi(x,\vec d)$ and $\phi(g(a),g(\vec d))$. Therefore one of $h$ or $h'$ is inconsistent.

  If $h$ is inconsistent then $\bot\in\lim_n h$, and therefore $T\cup\{\phi(a,\vec d)\}\cup\Sigma$ must be inconsistent. But then there is a finite subset $\Sigma_0$ of $\Sigma$ which implies $\forall x\,\neg\phi(a,\vec d)$, and therefore $\bigvee\neg\Sigma_0$ is in $\Gamma$, which is a contradiction.

  Therefore $h'$ is inconsistent, so at some finite stage, we deduce a contradiction from $T\cup\Gamma\cup\{\forall x\,\neg\phi(x,\vec d)\}$, and therefore the finitely many elements of $\Gamma$ which have so far been outputted are equivalent to $\exists x\,\phi(x,\vec d)$.
\end{proof}

It is useful to imagine this proof as a competition between $h$, which is trying to justify $\phi(a,\vec d)$, and $h'$, which is trying to justify $\forall x\,\neg\phi(x,\vec d)$. The approximation $h'$ is ``trying'' to prevent $h$ from consistently concluding $\phi(a,\vec d)$, but every time it forces $h$ to output $\bot$, $h$ discovers a new element $\bigvee\Sigma_0$ in $\Gamma$, which forces $h'$ to adjust its output.

The argument above is merely a rephrasing of the usual embedding test, but can now be expressed in purely arithmetic terms, making it amenable to the usual methods of proof mining. The real content of any proof of quantifier elimination using such a test is in the proof that models satisfy the corresponding embedding property, which we can then reinterpret as statements about these approximations of models.

To illustrate this line of thinking, we briefly how we might apply it the proof of quantifier elimination for algebraically closed fields. (Algorithms for quantifier elimination in this case are well known\cite{MR0044472}, and indeed, the algorithm we obtain from proof mining is precisely the standard one.)

The embedding test argument goes roughly as follows: let $M, N$ be algebraically closed fields, let $A\subseteq M$, let $f:A\rightarrow N$ be an embedding in the language of rings, and let $a\in M\setminus A$. (Note that this means that $f$ is an embedding of the ring generated by $A$ into $N$.)

We consider three cases. First, if $a$ is in the field generated by $A$ then $a=b/c$ where $b,c$ are in the ring generated by $A$, so we must define $g(a)=f(b)/f(c)$. Second, if $a$ is not in the field generated by $A$ but is algebraic over this field then we take a minimal polynomial $p$ for $a$ and take $g(a)$ to be any root of the polynomial $f(p)$. Finally, if $a$ is not algebraic over the field generated by $A$ then $a$ is transcendental, so we take any transcendental element of an elementary extension of $N$.

To try to extract an algorithm from this argument, we could reason as follows. Let $\exists x\,\phi(x,\vec d)$ be the formula we are considering, with $\phi$ quantifier-free. $M$ and $N$ are replaced by our approximations $h$ and $h'$. $h'$ has to go first, attempting to build the algebraically closed field $N$, which it may as well take to be the algebraic closure of $\vec d$, since this is the smallest choice, and therefore minimizes the risk of being forced to make $\phi(x,\vec d)$ true for any $x$.

As $h'$ considers various elements of the algebraic closure, it always has to decide that $\neg\phi(x,\vec d)$ holds---that is, $h'$ keeps outputting sentences $p(b,\vec d)=0\wedge\neg\phi(b,\vec d)$ for various polynomials $p$ and elements $b$. Whatever consequences this has for quantifier-free statements only involving $\vec d$ go into $\Sigma$, and therefore $h$ has to output them as well. $h$ is attempting to satisfy $\exists x\,\phi(x,\vec d)$, and is forced into the transcendental case from the model theoretic argument---$h'$ controls what happens to algebraic elements, so the only way $h$ can try to make $\phi(a,\vec d)$ true is by taking $a$ to be transcendental.

The model theoretic argument tells us that all transcendental elements behave the same way---that when $a,a'$ are two elements transcendental over $\vec d$, $\phi(a,\vec d)$ is true exactly when $\phi(a',\vec d)$. Compactness says the same must be true when $a,a'$ are ``almost transcendental''---there must be a finite set of polynomials in $\vec d$ so that if $a,a'$ satisfy none of them, $a,a'$ are equivalent $\phi(a,\vec d)$ is true exactly when $\phi(a',\vec d)$.  In this case, these polynomials are easy to identify: it suffices to determine that $p(a,\vec d)\neq 0$ for every polynomial $p$ appearing in $\phi$.

So once $h'$ has declared that no root of any polynomial appearing in $\phi$ is allowed to satisfy $\phi(x,\vec d)$, one of two things happen. First, it could be that, after replacing every equality in $\phi$ with $\bot$, we get a true statement---that is, $\phi(x,\vec d)$ is satisfied by sufficiently transcendental elements. Then $h'$ is inconsistent, because the model being approximated by $h'$ contains sufficiently transcendental element, because algebraically closed fields are always infinite.

Otherwise, $h'$ has briefly scored a victory over $h$, having ruled out all possible solutions to $\phi(x,\vec d)$. This victory is in the form of a collection of quantifier-free formulas
\[\bigwedge_{1\leq i<j\leq d}b_i\neq b_j\wedge \bigwedge_{1\leq i\leq d}p(b_i,\vec d)\wedge\neg\phi(b_i,\vec d)\]
which $h'$ has output. The model theoretic argument tells us that there must be some quantifier-free formula in just $\vec d$ which expresses this. This is the most difficult (and interesting) step: the model theoretic argument invokes some standard facts from field theory (about the number of roots of a polynomial and the fact that all roots of the minimal polynomial satisfy the same formulas). In order to obtain a quantifier-elimination algorithm, we would have to work through the actual proofs of these facts to identify what the corresponding formula is; the methods of proof-mining give us a systematic way to do this. (As mentioned above, if we actually did this step, we would rediscover the usual algorithm.)

In either case, we have discovered new statements in $\Gamma$, which $h'$ is then forced to accept, forcing $h'$ to now go through finitely many cases involving the possibility that a root of one of these polynomials satisfies $\phi(x,\vec d)$, and the argument continues.

In algorithmic terms, we have replaced a split between the transcendental case and the algebraic case with a split between a ``sufficiently transcendental'' case and an ``algebraic with a small bound'' case. Being transcendental is essentially a $\Pi_1$ property in this context---it amounts to saying that there is an infinite list of polynomials which $a$ does not satisfy, so we have replaced a choice between a $\Pi_1$ case and a $\Sigma_1$ case with a more constructive choice between a \emph{bounded} $\Sigma_1$ case and its negation. This is a typical situation when proof mining: we replace non-effective case splits with effective approximations. (If the split were more complicated---say, between a $\Pi_2$ case and a $\Sigma_2$ case---the corresponding constructive approximation would also more complicated.)

In the remainder of this paper, we illustrate this idea by using the ideas described above to extract algorithms from the two quantifier elimination results from \cite{MR2235481}.

\section{Algebraically Closed Fields with Small Subgroups}

In \cite{MR2235481}, van den Dries and G\"unaydin consider the theories of algebraic and real closed fields with a distinguished small subgroup and, using saturated model arguments, give ``relative'' quantifier elimination results showing that all formulas are equivalent to Boolean combinations of formulas of a particular simple form. In this section and the next, we apply the ideas of the previous section to give explicit algorithms corresponding to each of these results.

In this section, we let $\mathcal{L}$ be the language of rings, $\{0,1,+,-,\cdot\}$. We are initially interested in the expanded language $\mathcal{L}(U)$, where $U$ is a unary predicate symbol which will pick out our distinguished subgroup. A $\mathcal{L}(U)$-structure has the form $(K,G)$ where $K$ is a field and $G\subseteq K$ is a group.

\begin{definition}
  $G$ is \emph{small} in $K$ if whenever $f:K^m\rightarrow K$ is definable, $f(G^m)\subsetneq K$.
\end{definition}
That is, no definable function of a power of $G$ is sufficient to ``cover'' all of $K$.

The theory $\mathrm{ACF}(\mathfrak{m})$ is the $\mathcal{L}(U)$-theory saying that $K^\times$ is an algebraically closed field and $G\subseteq K$ is a group which is small in $K$.

$\mathrm{ACF}(\mathfrak{m})$ is not complete, since it does not specify much about the group $G$. So what we can hope to show is that the theory of $(K,G)$ reduces to the theory of $G$ in some way. We cannot quite reduce to the theory of $G$ as a group, because the additive structure of $K$ can be see on $G$. To solve this, van den Dries and G\"unaydin consider the language $\mathcal{L}(U,\Sigma)$ which adds, for every $\vec k\in\mathbb{Z}^n$, a new $n$-ary relation symbol $\Sigma_{\vec k}$, whose intended interpretation is the set
\[\{(g_1,\ldots,g_n)\in G^n\mid k_1g_1+\cdots+k_ng_n=0\}.\]

Then $\ACFm$ is the theory which extends $\mathrm{ACF}(\mathfrak{m})$ by the axioms
\[\forall \vec y\ (\Sigma_k(\vec y)\leftrightarrow \bigwedge_{i\leq n}Uy_i\wedge k_1y_1+\cdots+k_ny_n=0).\]

$\mathcal{L}(U,\Sigma)$ contains a sublanguage $\mathcal{L}_{\mathfrak{m}}(\Sigma)=\{1,\cdot\}\cup\{\Sigma_{\vec k}\mid \vec k\in\mathbb{Z}^n\}$, the language of multiplicative monoids with additive relations. We think of $\mathcal{L}_{\mathfrak{m}}(\Sigma)$ as the ``natural'' language of $G$ for our purposes, and embed it into $\mathcal{L}(U,\Sigma)$ as follows.

\begin{definition}
  We abbreviate quantification bounded to $U$ by $\exists^U$ and $\forall^U$---that is, $\exists^U\vec y\,\phi$ abbreviates $\exists \vec y(\bigwedge_{i\leq |\vec y|}Uy_i\wedge\phi)$ and $\forall^U\vec y\,\phi$ abbreviates $\forall \vec y(\bigwedge_{i\leq |\vec y|}Uy_i\rightarrow\phi)$.
  
  A \emph{$U$-formula} is a formula in the language $\mathcal{L}(U,\Sigma)$ defined inductively by:
  \begin{itemize}
  \item atomic formulas in $\mathcal{L}_{\mathfrak{m}}(\Sigma)$ are $U$-formulas,
  \item if $\theta$ is a $U$-formula, so is $\neg\theta$,
  \item if $\theta,\theta'$ are $U$-formulas, so are $\theta\wedge\theta'$ and $\theta\vee\theta'$,
  \item if $\theta$ is a $U$-formula, so are $\exists^Ux\,\theta$ and $\forall^Ux\,\theta$.
  \end{itemize}

  A \emph{$KG$-formula} is a formula of the form 
  \[\exists^U\vec y(\theta(\vec y)\wedge\psi(\vec x,\vec y))\]
  where $\theta$ is a $U$-formula and $\psi$ is a quantifier-free $\mathcal{L}$-formula.
\end{definition}
Since $\mathcal{L}$-formulas in $\mathrm{ACF}$ admit quantifier elimination, we may of course assume that the $\psi$ in a KG-formula is quantifier-free. 

In this setting, van den Dries and G\"unaydin show the following.
\begin{theorem}[\cite{MR2235481}[Theorem 3.8]]
  Each $\mathcal{L}(U,\Sigma)$-formula $\phi(\vec x)$ is equivalent in $\ACFm$ to a Boolean combination of KG-formulas.
\end{theorem}

In this section, we give an explicit algorithm for this. For intermediate steps we need a related but more complicated kind of formula. The main technical step, which our formalism is essentially devised to accommodate, is the following.

Imagine we are working in a model $(K,G)$ of $\ACFm$, and we consider a polynomial $p(\vec x,\vec u)$ where the $\vec x$ are arbitrary parameters from $K$. This polynomial defines a subset of $G$, $\{\vec u\in G^{|\vec u|}\mid p(\vec x,\vec u)=0\}$. We wish to find a formula $\rho$ which defines the same subset of $G$ without using the parameters $\vec x$. To make this possible, we will allow $\rho$ to make reference to some new parameters, but these new parameters must belong to $G$.

For instance, when $|\vec u|=1$, there are at most as many roots as the degree of $p$ in $u$, so we can simply list them:
\[\{u\in G\mid p(\vec x,u)=0\}=\{u\in G\mid u=v_1\vee u=v_2\vee\cdots \vee u=v_k\}\]
for some $k$ and some choice of $v_1,\ldots,v_k$.

Note that the formula $u=v_1\vee u=v_2\vee\cdots \vee u=v_k$ is not quite uniform in $\vec x$---for different values of $\vec x$, the number of solutions in $G$ might be different. (This can be addressed in the simple example by allowing the $v_i$ to duplicate, but in general the non-uniformity is more complicated.) We can express this non-uniformity in a fairly simple way by making a case distinction based on $\vec x$ which can be expressed by a formula involving $\vec x$. Say $p$ has degree $m$ in $u$, so $k\leq m$. Then:
\begin{itemize}
\item suppose there are $m$ distinct roots of $p(\vec x,y)=0$ in $G$; then for any list $v_1,\ldots,v_m$ of distinct roots of $p(\vec x,y)=0$ in $G$, $u$ is a root if it is equal to some $v_i$,
\item suppose not, but suppose but there are $m-1$ distinct roots of $p(\vec x,y)=0$ in $G$; then for any list $v_1,\ldots,v_{m-1}$ of distinct roots of $p(\vec x,y)=0$ in $G$, $u$ is a root if it is equal to some $v_i$, 
\item suppose not \ldots,
\item suppose there are not $2$ distinct roots, but there is one $1$ root of $p(\vec x,y)=0$ in $G$; then for root $v_1$ of $p(\vec x,y)=0$ in $G$, $u$ is a root if it is equal to $v_1$,
\item if there is not even one root of $p(\vec x,y)=0$ in $G$, then $u$ is never a root of $p(\vec x,y)=0$.
\end{itemize}

That is:
\begin{itemize}
\item there is a finite list of possible cases,
\item the cases are determined by formulas of the form $\exists^U \vec v\, \psi(\vec x,\vec v)$,
\item if the formula $\psi(x,\vec v)$ determines which case we are in, there is a formula $\rho(\vec v,u)$ so that \emph{any} choice of $\vec v$ in $U$ satisfying $\psi(\vec x,\vec v)$ has the property that $\rho(\vec v,u)$ is equivalent to $p(x,u)=0$.
\end{itemize}
In the simple case, the formulas $\psi$ and $\rho$ were in $\mathcal{L}$; in general, we need slightly more complicated formulas. Moreover, we will need to iterate this process, so we need to deal with the case where $\vec x$ is itself a mix of parameters from $K$ and parameters which have already been restricted to come from $G$.

Since we will be dealing with these case splits frequently, it will be helpful to have a notation for them.
\begin{definition}
A formula $\phi(\vec x;\vec y)$ with the given division of variables is \emph{split} if it is a Boolean combination of $U$-formulas whose only variables are $\vec y$ and  $\mathcal{L}$-formulas which may contain all the variables.
  
Let $J$ be a finite set and let $\{\eta^+_j(\vec x;\vec u,\vec y_j)\}_{j\in J}$, $\{\eta^-_j(\vec x;\vec u,\vec y'_j)\}_{j\in J}$ be split formulas. We say these are a \emph{partition of cases} if $\ACFm$ proves that for every $\vec x,\vec u$ there is exactly one $j$ so that $\exists^U\vec y_j\,\eta^+_j(\vec x;\vec u,\vec y_j)\wedge\forall^U\,\vec y'_j\eta^-_j(\vec x;\vec u,\vec y'_j)$. We will often write $\vec y_j$ instead of $\vec y'_j$ for the variables in $\eta^-_j$, since they never appear together with the variables in $\vec y_j$, but note that there is no assumption that these two lists have the same length.

  When $\{\psi_j(\vec x;\vec u,\vec y_j)\}_{j\in J}$ is a collection of formulas, we say it is \emph{rigid} under a partition of cases $\{\eta^+_j(\vec x;\vec u,\vec y_j)\}_{j\in J}$, $\{\eta^-_j(\vec x;\vec u,\vec y_j)\}_{j\in J}$ if, for each $j$, $\ACFm$ proves
  \[\forall\vec x\left[\left(\exists^U\vec y_j\,\eta^+_j(\vec x;\vec u,\vec y_j)\wedge\forall^U\vec y'_j\,\eta^-_j(\vec x;\vec u,\vec y'_j)\right)\rightarrow\forall^U\vec y_j\vec y'_j\,\left(\psi_j(\vec x;\vec u,\vec y_j)\leftrightarrow\psi_j(\vec x;\vec u,\vec y'_j)\right)\right].\]
  That is, whenever $\vec x,\vec u$ belongs in the case determined by $j$, whether $\psi_j$ holds is independent of the choice of witnesses $\vec y_j$.

   When $\{\eta^+_j(\vec x;\vec u,\vec y_j)\}_{j\in J}$, $\{\eta^-_j(\vec x;\vec u,\vec y'_j)\}_{j\in J}$ is a partition of cases and $\{\psi_j(\vec x;\vec u,\vec y_j)\}_{j\in J}$ is a rigid collection of formulas, we write
   \[\bigoplus_{j\in J}\eta^{\pm}_j(\vec x;\vec u,\vec y_j)\rightarrow\psi_j(\vec x;\vec u,\vec y_j)\]
   to abbreviate the formula
    \[\bigvee_{j\in J}\exists^U\vec y_j\,\eta^+_j(\vec x;\vec u,\vec y_j)\wedge\forall^U\vec y'_j\,\eta^-_j(\vec x;\vec u,\vec y'_j)\wedge\psi_j(\vec x;\vec u,\vec y_j).\]

   A \emph{controlled formula} is a formula of the form 
   \[\bigoplus_{j\in J}\eta^{\pm}_j(\vec x;\vec u,\vec y_j)\rightarrow\psi_j(\vec x;\vec u,\vec y_j)\]
   where each each $\psi_j(\vec x;\vec u,\vec y_j)$ is a split formula.

  \end{definition}
  

  The negation of a controlled formula can be expressed as a controlled formula: $\neg\bigoplus_{j\in J}\eta_j^{\pm}(\vec x;\vec u,\vec y_j)\rightarrow\psi_j(\vec x;\vec u,\vec y_j)$ is equivalent to $\bigoplus_{j\in J}\eta_j^{\pm}(\vec x;\vec u,\vec y_j)\rightarrow\neg\psi_j(\vec x;\vec u,\vec y_j)$. (This is why we need to be sure that each $\vec x$ can belong to only one case.)

  A conjunction of controlled formulas is also equivalent to a controlled formula:
  \[(\bigoplus_{j\in J_0}\eta_j^{\pm}(\vec x;\vec u,\vec y_j)\rightarrow\psi_j(\vec x;\vec u,\vec y_j))\wedge(\bigoplus_{j\in J_1}\theta_j^{\pm}(\vec x;\vec u,\vec y'_j)\rightarrow\nu_j(\vec x;\vec u,\vec y'_j))\]
  is equivalent to
  \[\bigoplus_{(j_0,j_1)\in J_0\times J_1}\eta_{j_0}^{\pm}(\vec x;\vec u,\vec y_j)\wedge\theta_{j_1}^{\pm}(\vec x;\vec u,\vec y'_j)\rightarrow\psi_j(\vec x;\vec u,\vec y_j)\wedge\nu_j(\vec x;\vec u,\vec y'_j).\]
  Of course it follows that the disjunction of two controlled formulas is equivalent to a controlled formulas as well.



Our first goal is to show that any definable subset of $G$ is already definable using parameters from $G$. The choice of definition is not uniform in the parameters, but we will show that the non-uniformity can be described by a controlled formula. The main step is dealing with the case where the definition is given by a single polynomial, and the main step in that is showing that we can replace a single parameter from the larger model with parameters from $G$, which is the content of the following lemma.
 
\begin{lemma}\label{thm:split_equality}
  Let $p(\vec x,z;\vec u)$ be a polynomial. Then $\ACFm$ proves that for any $\vec x,z$ and any $\vec u$ satisfying $U$, $p(\vec x,z;\vec u)$ is equivalent to a controlled formula of the form
\[\bigoplus_{j\in J}\eta^{\pm}_j(\vec x,z;\vec y_j)\rightarrow\psi_j(\vec x;\vec u,\vec y_j).\]
\end{lemma}
Note that which variables appear in which formulas is crucial here. The formulas $\eta^{\pm}_j$ depend on $\vec x,z$, but not on $\vec u$, while $\psi_j$ depends on $\vec x,\vec u$, but not on $z$. That means the algebraic set $\{\vec u\in U\mid p(\vec x,z,\vec u)=0\}$ defined using the extra parameter $z$ can instead be defined using parameters $\vec y_j$ from $U$.

Roughly speaking, we should think of this lemma as follows: the parameters $\vec x,z$ come from some small model $(K',G')$ contained in the large model $(K,G)$. This lemma amounts to showing the algebraic disjointness of $K'$ from the field generated by $G$ over the field generated by $\vec x\cup G''$.
\begin{proof}
  We proceed by main induction on $|\vec u|$; when $|\vec u|=0$, take $J=\{1,2\}$, $\eta^+_1$ to be $p(\vec x,z)=0$ and $\psi_1(\vec x)$ to be $\top$, and take $\eta^+_2$ to be $p(\vec x,z)\neq 0$ and $\psi_2(\vec x)$ to be $\bot$.

  Suppose $|\vec u|>0$ and write $p(\vec x,z,\vec u)=\sum_{i\leq t}p_i(\vec x,\vec u)z^i$. We proceed by side induction on the number of coefficients $i\leq t$ so that $p_i$ is not the constant term $0$. If there are no such coefficients then we may take $J=\{1\}$, $\eta^+_1$ to be $\top$, and $\psi_1$ to be $\top$.

  Otherwise, let $i_0\leq t$ be least such that $p_{i_0}$ is not the constant term $0$. We will split into two cases, essentially based on whether there are any solutions where $p_{i_0}(\vec x,\vec u)\neq 0$. We will ultimately take $J=J_0\cup J_1$ where, for every $j\in J_0$, $\eta^+_j(\vec x,\vec z,\vec y_j)$ will imply that, for some $\vec y'_j\subseteq\vec y_j$, $p(\vec x,z,\vec y'_j)= 0\wedge p_{i_0}(\vec x,\vec y'_j)\neq 0$, while for every $j\in J_1$, $\eta^-_j$ will imply $\forall^U\vec u\, (p(\vec x,z,\vec u)=0\rightarrow p_{i_0}(\vec x,\vec u)=0)$.

  To find $J_1$ and the corresponding formulas, we apply the side inductive hypothesis to $\sum_{i\leq t, i\neq i_0}p_i(\vec x,\vec u)z^i$, so this is equivalent to some controlled formula
  \[\oplus_{j\in {J_1}}\theta^{\pm}(\vec x,z;\vec y_j)\rightarrow\psi_j(\vec x;\vec u,\vec y_j).\]
For $j\in J_1$, we take $\eta^+_{j_1}$ to be $\theta^+$ and $\eta^-_{j_1}(\vec x,z;\vec y_j,\vec y'_j)$ to be $\theta^-_{j_1}(\vec x,z;\vec y_j)\wedge (p(\vec x,z;\vec y'_j)=0\rightarrow p_{i_0}(\vec x;\vec y'_j))$.

  It remains to handle the case where $\exists^U\vec y\,(p(\vec x,z,\vec y)=0\wedge p_{i_0}(\vec x,\vec y)\neq 0)$.

  Since $|\vec u|\geq 1$, let us separate out the last element of $\vec u$, so $\vec u=\vec u_0,v$ and view $p$ as a polynomial in $v$, $p(\vec x,z,\vec u_0,v)=\sum_{s\leq t'}p^*_s(\vec x,z,\vec u_0)v^s$. We apply the main inductive hypothesis to each $p^*_s$, obtaining sets $J_s$ and corresponding formulas $\eta^\pm_{s,j}$ and $\psi_{s,j}$. We will take $J_0=\prod_sJ_s$, $\eta^+_{\{j_s\}}$ to be $\bigwedge_j\eta^{\pm}_{s,j_s}(\vec x,z;\vec y_{j_s})\wedge p(\vec x,z;\vec y)=0\wedge p_{i_0}(\vec x;\vec y)\neq 0$. We take $\eta^-_{\{j_s\}}$ to be $\bigwedge_j\eta^-_{s,j_s}$. It remains to describe the formula $\psi_{\{j_s\}}$.

For the remainder of the argument, we assume we have fixed some choice of $\{j_s\}$ and some choice of the corresponding witnesses $\vec y_{j_1},\ldots,\vec y_{j_s}$ and, in particular, some fixed choice of $\vec y$ so that $p(\vec x,z;\vec y)=0$ while $p_{i_0}(\vec x;\vec y)\neq 0$. 

So fix some choice of $\{j_s\}$ and some choice of the corresponding witnesses $\vec y_{1,j_1},\ldots,\vec y_{s,j_s},\vec y$.  Let $p'(\vec x,z,\vec y)=p(\vec x,z,\vec y)/z^{i_0}$. Informally, the idea will be as follows. Suppose we begin with a field containing $\vec x,\vec y,\vec u_0$. Adding $z$, as a root of $p'(\vec x,z,\vec y)$, gives a field extension of finite degree, and adding $v$ as a root of $p(\vec x,z,\vec u_0,v)$ is a further finite field extension.

Standard arguments about fields show that if we instead added $v$ followed by $z$, we would get an extension of the same degree; in particular, it follows that $v$ must satisfy a polynomial in $\vec x,\vec y,\vec u_0$. The remainder of our algorithm consists of working through that argument to identify the polynomial satisfied by $v$.

When $z$ and $v$ are added as roots like this, $\{v^iz^j\}_{0\leq i<t',0\leq j<t}$ forms a spanning set of size $tt'$ for the field generated by adding $v$ and $z$. (The choice of $tt'$ as the upper bound may be too high to be a basis, but we only need that this is a spanning set.) In particular, since there exists a spanning set of size $tt'$, the $tt'+1$ elements $1,v,v^2,\ldots,v^{tt'}$ cannot be linearly independent (over the field already containing $\vec x,\vec y,\vec u_0$), so we will be able to show that $v$ satisfies a polynomial in $\vec x,\vec y,\vec u_0$ of degree at most $tt'$.

First, we split into cases based on the formulas $\psi_{s,j_s}(\vec x,\vec u_0)$---that is, our formula $\psi_{\{j_s\}}(\vec x,\vec u_0,v)$ will have the form
\[\bigvee_{S\subseteq[0,t']}\bigwedge_{s\in S}\psi_{s,j_s}(\vec x,\vec u_0)\wedge\bigwedge_{s\not\in S}\neg\psi_{s,j_s}(\vec x,\vec u_0)\wedge \psi'_{s,S}(\vec x,\vec u_0,v).\]
So let us suppose we have chosen some set $S\subseteq[0,t']$, determining which $\psi_{s,j_s}$ are true, and therefore which $p^*_s(\vec x,z,\vec u_0)$ are non-zero.

Since $p'(\vec x,z,\vec y)=0$, for each $s$ so that $p^*_s(\vec x,z,\vec u_0)$ is non-zero, there is a term $\frac{p^+_s(\vec x,z,\vec u_0,\vec y)}{p^-_s(\vec x,\vec u_0,\vec y)}$ so that
\[p^*_s(\vec x,z,\vec u_0) \frac{p^+_s(\vec x,z,\vec u_0,\vec y)}{p^-_s(\vec x,\vec u_0,\vec y)}=1.\]

Let $S_0=\{v^iz^j\}_{0\leq i<t,0\leq j<t}$. We will proceed inductively: at each stage $n$, we will have a spanning set $S_n\subseteq S_0\cup\{v^i\}_{0\leq i<tt'}$ of size $tt'$ and, for each element $s$ of $S_0\setminus S_n$, we will witness that $S_n$ is still a spanning set by keeping track of a ratio of linear combinations
\[s=\sum_{w\in S_n}\frac{q_{n,w}(\vec x,\vec y,\vec u_0)}{q'_{n,w}(\vec x,\vec y,\vec u_0)}w\]
where the denominator is non-zero.

Suppose we have constructed $S_n$. Let $i\leq tt'$ be least so that $v^i\not\in S_n$. Consider the polynomial $v^{i'}p(\vec x,z,\vec u_0,v)$ where $i'$ is chosen so that $v^i$ appears but no higher power of $v$ does. (Note that whether $p^*_s(\vec x,z,\vec u_0)=0$ is determined by $\psi_{s,j_s}(\vec x,\vec u_0,\vec y_s)$, so what $i'$ is does not depend on $z$.)

Viewing $v^{i'}p(\vec x,z,\vec u_0,v)$ as a linear combination of terms from $S_0$ with coefficients in $\vec x,\vec u_0$, we isolate all terms containing $v^i$:
\[v^ip^*_{i-i'}(\vec x,z,\vec u_0)=\sum_{w\in S_0}q_w(\vec x,\vec u_0)w.\]
Multiplying both sides by $\frac{1}{p^*_{i-i'}(\vec x,z,\vec u_0)}=\frac{p^+_{i-i'}(\vec x,z,\vec u_0,\vec y)}{p_{i-i'}(\vec x,\vec u_0,\vec y)}$, we obtain
\[v^i=\sum_{w\in S_0}\frac{q'_w(\vec x,\vec u_0,\vec y)}{r'_w(\vec x,\vec u_0,\vec y)}w.\]
Since each element of $S_0$ is a linear combination of elements from $S_n$, we can further rewrite this as
\[v^i=\sum_{w\in S_n}\frac{q''_w(\vec x,\vec u_0,\vec y)}{r''_w(\vec x,\vec u_0,\vec y)}w.\]
If every $w$ such that $q''_w(\vec x,\vec u_0,\vec y)\neq 0$ is a power of $v$ then, after multiplying through by the denominators, we have obtained a polynomial in $v$. 

Otherwise, there is some $w\in S_n$ of the form $v^{i'}z^j$ whose coefficient is non-zero; therefore we can isolate this $w$ to obtain an expression for $w$ in terms of $S_n\setminus\{w\}\cup\{v^{i}\}$, so we may take $S_{n+1}=S_n\setminus\{w\}\cup\{v^{i}\}$.

Note that if $n=tt'$, we cannot be in the second case: by the time $n=tt'$, we have encountered the polynomial we are looking for. Our formula $\psi_{s,S}$, then, is a further large disjunction over whether various polynomials in $\vec x,\vec u_0,\vec y$ are zero, with each disjunct concluding that $p(\vec x,z,\vec u_0,v)=0$ exactly when the final polynomial $\prod_w r''_w(\vec x,\vec u_0,\vec y)v^i=\sum_{w\in S_n}q''_w(\vec x,\vec u_0,\vec y)w$ obtained in that disjunct holds.
\end{proof}

\begin{cor}\label{thm:replace_one_var}
  Let $\phi(\vec x,z;\vec u)$ be an $\mathcal{L}$-formula. Then $\ACFm$ proves that $\phi$ is equivalent to a controlled formula
 \[\bigoplus_{j\in J}\eta^{\pm}_j(\vec x,z;\vec y_j)\rightarrow\psi_j(\vec x;\vec u,\vec y_j).\]

\end{cor}
\begin{proof}
  Since $\phi$ is a $\mathcal{L}$-formula and ACF has quantifier elimination, this follows by replacing $\phi$ with a quantifier-free formula in which all atomic formulas have the form form $p=0$ for some $t$, applying Lemma \ref{thm:split_equality} to the resulting polynomials, and then combining the formulas given by the lemma.
\end{proof}

\begin{cor}
  Let $\phi(\vec x, z;\vec u)$ be a split formula. Then $\ACFm$ proves that $\phi$ is equivalent to a controlled formula
  \[\bigoplus_{j\in J}\eta^{\pm}_j(\vec x, z;\vec y_j)\rightarrow\psi_j(\vec x;\vec u,\vec y_j).\]
\end{cor}
\begin{proof}
  Apply the previous corollary to each $\mathcal{L}$-formula in $\phi$. Since $z$ does not appear in the $U$-formulas, replacing each $\mathcal{L}$-formula with an equivalent controlled formula and then combining the resulting controlled formulas into a single controlled formula gives the desired formula.
\end{proof}

Finally we extend to the case where $z$ is replaced by a tuple.

\begin{lemma}\label{thm:replace_var}
  Let $\phi(\vec x,\vec z;\vec u)$ be a split formula. Then $\ACFm$ proves that $\phi$ is equivalent to a controlled formula
  \[\bigoplus_{j\in J}\eta^{\pm}_j(\vec x,\vec z;\vec y_j)\rightarrow\psi_j(\vec x;\vec u,\vec y_j).\]

\end{lemma}
\begin{proof}
  By induction on $|\vec z|$. When $|\vec z|=0$, the statement is trivial since we may take $\psi_0$ to be $\phi$.

  If $|\vec z|=n+1$, we may write $\vec z=\vec z_0,z$. Applying the previous corollary gives formulas so that $\ACFm$ proves $\phi(\vec x,\vec z_0,z;\vec u)$ is equivalent to
  \[\bigoplus_{j\in J}\eta^{\pm}_j(\vec x,\vec z;\vec y_j)\rightarrow\psi_j(\vec x,\vec z_0;\vec u,\vec y_j).\]

We then apply the inductive hypothesis to each $\psi_j$, obtaining controlled formulas
\[\bigoplus_{j'\in J_j}\eta^{\pm}_{j,j'}(\vec x,\vec z;\vec y_j,\vec y_{j'})\rightarrow\psi_{j,j'}(\vec x,\vec z_0;\vec u,\vec y_j,\vec y_{j'}).\]
We could try to combine these by setting $\theta^{\pm}_{j,j'}(\vec x,\vec z;\vec y_j,\vec y_{j'})$ to be $\eta^{pm}_j(\vec x,\vec z;\vec y_j)\wedge\eta^{\pm}_{j,j'}(\vec x,\vec z;\vec y_j,\vec y_{j'})$, but this does not quite give a partition of cases: for each $\vec x,\vec z,\vec y_j$, there's a unique case $j'$, it could be that different choices of $\vec y_j$ lead to different choices of $j'$.

This forces us to take a more complicated approach, where our partition of cases insists on a choice of witnesses $\vec y_j$ so that the ``first possible case'' is chosen at the next step. We need to define the first possible case a little carefully to deal with the restricted form in which we state cases.

Assume that $\vec x,\vec z$ belongs to some case $j$---that is, $\exists^U\vec y_j\,\eta^+(\vec x,\vec z;\vec y_j)\wedge\forall^U\vec y_j\,\eta^-(\vec x,\vec z;\vec y_j)$ holds.  For each choice of witnesses $\vec y_j$ to $\eta^+$, there is exactly one case $j'(\vec y_j)\in J_j$ that holds, and all others fail. That means that, for all $j''\neq j'$, at least one of $\forall^U\vec y'_{j''}\,\eta^+_{j,j''}(\vec x,\vec z;\vec y_j,\vec y_{j''})$ and $\exists^U\vec y'_{j''}\,\eta^-_{j,j''}(\vec x,\vec z;\vec y_j,\vec y_{j''})$ holds. In order to name a unique case, we need to pay attention not only to which value of $j'$ we need, but to which of the two ways the other cases are ruled out.

We will let $K_j$ be the set of all pairs $(m,j')$ where $j'\in J_{j}$ and $m\leq|J_{j}|-1$. We order these pairs lexicographically. We will now define $\theta^{\pm}_{j,m,j'}$ to state that the case $j$ holds for $\vec x,\vec z$, and that among all possible $\vec y_j$ witnessing this, $(m,j')$ is least such that the case $j'$ holds of $\vec x,\vec z,\vec y_j$ with $|J_{j}|-1-m$ of the values of $j''\in J_{j}\setminus\{j'\}$ ruled out by finding witnesses to $\eta^-$. For each $\vec x,\vec z$, there is a unique choice of $j$ and least $(m,j')$ so this occurs, so it suffices to show that we can express this with a pair of formulas of the specified kind.

We take $\theta^+_{j,m,j'}(\vec x,\vec z;\vec y_j,\{\vec y_{j''}\}_{j''\in J_j})$ to state:
\begin{itemize}
\item $\eta^+_j(\vec x,\vec z;\vec y_j)$ holds,
\item $\eta^+_{j,j'}(\vec x,\vec z;\vec y_j,\vec y_{j'})$ holds,
\item there is a set of $j''\in J_j\setminus\{j'\}$ of size $m$ for which $\eta^-_{j,j''}(\vec x,\vec z;\vec y_j,\vec y_{j''})$ holds,
\end{itemize}
and $\theta^-_{j,m,j'}(\vec x,\vec z;\vec y_j,\{\vec y_{j''}\}_{j''\in J_j})$ to state:
\begin{itemize}
\item $\eta^-_j(\vec x,\vec z;\vec y_j)$ holds,
\item for all $(m^*,j^*)<(m,j')$, $\theta^+_{j,m^*,j^*}(\vec x,\vec z;\vec y_j,\{\vec y_{j''}\}_{j''\in J_j})$ fails.
\end{itemize}

Suppose $\exists^U \vec y_j,\{\vec y_{j''}\}_{j''\in J_j}\, \theta^+_{j,m,j'}(\vec x,\vec z;\vec y_j,\{\vec y_{j''}\}_{j''\in J_j})$ and
\[\forall^U \vec y_j,\vec y'_j,\{\vec y_{j''}\}_{j''\in J_j}\, \theta^-_{j,m,j'}(\vec x,\vec z;\vec y_j,\vec y'_j,\{\vec y_{j''}\}_{j''\in J_j})\]
both hold. Then certainly there is some $\vec y_j$ witnessing that $\vec x,\vec z$ are in case $j$, and certainly $(m,j')$ is $\leq$ the first case at which $\vec x,\vec z,\vec y_j$ can be witnessed. Moreover, if $\forall^U\,\vec y_{j'}\,\eta^-(\vec x,\vec z;\vec y_j,\vec y_{j'})$ fails for any witnessing $\vec y_j$, since $j'$ is not the correct case for $\vec x,\vec z,\vec y_j$, some other case $j''$ holds, and therefore $\theta^+_{j,m-1,j''}$ would hold for some witnesses, contradicting $\theta^-_{j,m,j'}$.

At last, we can take $\psi_{j,m,j'}$ to be $\psi_j\wedge\psi_{j,j'}$.
\end{proof}




The next few lemmata show that we can eliminate a single existential quantifier. They correspond to the steps in the model theoretic proof, which considers several cases: first we consider the case where we can find the witness in $U$, then the case where we can find the witness in the algebraic closure, and finally the case where the witness is a transcendental element not in $U$.

\begin{lemma}\label{thm:controlled_U}
  Let $\phi(\vec z,u)$ be a formula which is equivalent to a controlled formula. Then there is a controlled formula equivalent to $\exists^Uu\,\phi(\vec z,u)$.



\end{lemma}
\begin{proof}
  Let $\bigoplus_{j\in J_0}\eta^{\pm}_j(\vec z;\vec y_j)\rightarrow\psi'_j(\vec z;u,\vec y_j)$ be a controlled formula equivalent to $\phi(\vec z,u)$. We apply Lemma \ref{thm:replace_var} to each $\psi'_j(\vec z;u,\vec y_j)$ and combine the resulting controlled formulas in the same way as in the proof of that lemma, obtaining an equivalent controlled formula $\bigoplus_{j'\in J_j}\eta^{\pm'}_{j,j'}(\vec z;\vec y_{j,j'})\rightarrow\psi'_{j,j'}(;u,\vec y_{j,j'})$ (with no variables in the left half of the split).

   Since the choice of case does not depend on $u$, it follows that $\exists^Uu\,\phi(\vec z;u)$ is equivalent to
   \[\bigoplus_{j'\in J_j}\eta^{\pm'}_{j,j'}(\vec z;\vec y_{j,j'})\rightarrow\exists^Uu\,\psi'_{j,j'}(;u,\vec y_{j,j'}).\]
Since the formula $\psi'_{j,j'}$ contains only variables from $U$, $\exists^Uu\,\psi'_{j,j'}(;u,\vec y_{j,j'})$ becomes a $U$-formula after replacing all atomic formulas of the form $\sum_{i\leq k}t_i=0$ with a corresponding formula $\Sigma(t_1,\ldots,t_k)$.
\end{proof}

The next lemma considers the case where $z$ is algebraic over the other variables.

\begin{lemma}\label{thm:controlled_root}
Suppose $\phi(\vec x,z)$ is equivalent to a controlled formula and let $\sum_{i\leq d}p_i(\vec x)z^i$ be a polynomial. Then $\exists z\,(\bigvee_{i\leq d}p_i(\vec x)\neq 0)\wedge (\sum_{i\leq d}p_i(\vec x)z^i=0)\wedge\phi(\vec x,z)$ is equivalent to a controlled formula.
\end{lemma}
\begin{proof}
  The first disjunct $\bigvee_{i\leq d}p_i(\vec x)\neq 0$ is equivalent to a controlled formula and does not depend on $z$, so we can have one case in our controlled formula asserting that the formula is false if $\bigwedge_{i\leq d}p_i(\vec x)=0$. 
  
  We focus on the other cases, where the polynomial $\sum_{i\leq d}p_i(\vec x)z^i=0$ is non-trivial. Let $\bigoplus_{j\in J}\eta^{\pm}_j(\vec x,z;\vec y_j)\rightarrow\psi_j(\vec x,z;\vec y_j)$ be the controlled formula equivalent to $\phi$. 

  We will consider pairs of functions $f,g:J\rightarrow\mathbb{N}$ such that $\sum_{j\in J}f(j)\leq d$ and $f(j)\leq g(j)\leq d$ for all $j$. We will abbreviate $\sum_{j\in J}f(j)$ by just $\sum f$.

  We will define a new controlled formula where $J'$ is the set of such pairs $f,g$. The idea is that the case corresponding to $f,g$ should be the case where, for each $j$, there are exactly $g(j)$ roots of $p(\vec x,z)=0$ which satisfy $\eta^+_j$, and exactly $f(j)$ of them also satisfy $\eta^-_j$.

  First, for each $j$ and $g(j)$, first consider the formula $\eta^+_{f,g}$ stating ``there are at least $g(j)$ roots which satisfy $\eta^+_j$ and at least $g(j)-f(j)$ of them do not satisfy $\eta^-_j$''. That is,
  \[\exists z_1,\ldots,z_{g(j)}\bigwedge_{i}p(\vec x,z_i)=0\wedge\bigwedge_{1\leq i<i'<g(z)}z_i\neq z_j\wedge\bigwedge_i\exists^U\vec y,\theta^+_j(\vec x,z_i,\vec y)\wedge\bigwedge_{i>f(j)}\exists^U\vec y\,\neg\theta^-_j(\vec x,z_i,\vec y).\]
  All the quantifiers are existential, so standard operations on first-order formulas (pulling existentials out of conjunctions, putting a formula in disjunctive normal form, and splitting the existential over the disjuncts) together with quantifier elimination for $\ACF$ (because the variables $z_i$ only appear in $\mathcal{L}$-formulas) allows us to rewrite this as a formula of the correct form.

Dually, we may take $\eta^-_{f,g}$ to state that, for each $j$, ``there are at most $\sum f$ roots, and at least $\sum f-g(j)$ of them do not satisfy $\eta^+_j$ while at least $\sum f-g(j)+f(j)$ of them satisfy either $\eta^-_j$ or $\neg\eta_j^+$''. We can rewrite this as a formula of the right kind in the same way. Observe that, taken together, $\exists^U\vec y\,\eta^+_{f,g}\wedge\forall^U\vec y\,\eta^-_{f,g}$ states that there are exactly $\sum f$ roots of $p(\vec x,z)=0$ and that for each $j$, there are exactly $g(j)$ of them satisfying $\eta^+_j$, while exactly $g(j)-f(j)$ of those also fail to satisfy $\eta^-_j$. This is precisely what we need, since it ensures there are exactly $f(j)$ roots in the $j$-th case.

To define $\psi_{f,g}$, we notice that our witnesses tell us precisely how to identify the $f(j)$ roots of $p(\vec x,z)=0$ which are in the $j$-th case, so in this case $\exists z\,p(\vec x,z)=0\wedge\phi(\vec x,z)$ is equivalent to stating that there are $\sum f$ values $z_i$ so that:
\begin{itemize}
\item each $z_i$ is a root of $p(\vec x,z)=0$,
\item the $z_i$ are all distinct,
\item they list the roots in the order of cases---that is, the first $f(1)$ roots are in the first case, the next $f(2)$ in the second case, and so on,
\item for some $i\leq \sum f$,  $\psi_j(\vec x,z_i,\vec y_j)$ holds when we take the $\vec y_j$ to be the suitable witnesses.
\end{itemize}
Once again we may write the inside of this formula in disjunctive normal form, distribute the existential over disjunctions, pull all $U$-formulas out of the quantifiers because they do not depend on the $z_i$, and then use quantifier-elimination for $\ACF$ to rewrite the $\mathcal{L}$-formulas as quantifier-free formulas.
\end{proof}

We finally combine these results to show the main quantifier elimination result of this section. 

\begin{theorem}
  Let $\phi(\vec x)$ be a $\mathcal{L}(U,\Sigma)$-formula. Then $\phi$ is equivalent to a controlled formula.
\end{theorem}
\begin{proof}
  By induction on $\phi$. If $\phi$ is $Ut$ for some term $t$, it is equivalent to $\exists^Uy\,\top \rightarrow y=t$. An equality $p(\vec x)=0$ is equivalent to $\top\rightarrow p(\vec x)=0$. We have discussed the case of connectives above.

  Consider a formula $\exists z\,\phi'(\vec x,z)$. By the inductive hypothesis, $\phi'(\vec x,z)$ is equivalent to a controlled formula $\bigoplus_{j\in J}\eta^{\pm}_j(\vec x,z;\vec y_j)\rightarrow\psi_j(\vec x,z;\vec y_j)$.
  
  Enumerate all polynomials appearing in any $\eta^{\pm}_j$ or $\psi_j$ which contain $z$, $p_1,\ldots,p_m$. Write each polynomial as a polynomial in $z$, $p_i(\vec x,z,\vec y)=\sum_{j\leq d_i}p_{i,j}(\vec x,\vec y)z^j$ and let $\theta_i(\vec x,z,\vec y)$ be the formula
  \[\left(\bigvee_{j\leq d_i}p_{i,j}(\vec x,\vec y)\neq 0\right)\wedge p_i(\vec x,z,\vec y)=0\]
  which says that $z$ is a root and also the polynomial is non-trivial.

 We split into cases---if there is a $z$, either it is a root of some $p_i$, or it isn't, so $\phi$ is equivalent to
  \[\exists z\left[\bigvee_{i\leq m}\exists^U\vec y\,\theta_i(\vec x,z,\vec y)\wedge\phi'(\vec x,z)\right]\vee\left[\bigwedge_{i\leq m}\forall^U\vec y\,\neg\theta_i(\vec x,z,\vec y)\wedge\phi'(\vec x,z)\right].\]

  Distributing the $\exists z$ over the disjunction, it suffices to show that each disjunct is equivalent to a controlled formula.

  Consider any of the first $m$ disjuncts. Rearranging, they have the form
  \[\exists^U\vec y\exists z\,p(\vec x,z,\vec y)=0\wedge \bigvee_{j\leq d_i}p_{i,j}(\vec x,\vec y)\neq 0\wedge\phi'(\vec x,z).\]
  Since $\bigvee_{j\leq d_i}p_{i,j}(\vec x,\vec y)\neq 0\wedge\phi'(\vec x,z)$ is equivalent to a controlled formula (it is a Boolean combination of $\phi'$, which is covered by the inductive hypothesis, and atomic formulas), we may apply Lemma \ref{thm:controlled_root} to show that
  \[\exists z\,p(\vec x,z,\vec y)=0\wedge \bigvee_{j\leq d_i}p_{i,j}(\vec x,\vec y)\neq 0\wedge\phi'(\vec x,z)\]
  is equivalent to a controlled formula, and then apply Lemma \ref{thm:controlled_U} $|\vec y|$ times to show that the whole disjunct is equivalent to a controlled formula.\footnote{This precisely corresponds to the model theoretic argument: one case is that we first add enough elements satisfying $U$ to make $z$ algebraic and then show that we can add $z$. On the syntactic side, these steps happen in the opposite order: we first reduce to a statement involving some algebraic after adding elements of $U$, and then eliminate the extra quantifiers over $U$.}

  Consider the final disjunct. Let $\eta^{\pm*}_j$ and $\psi^*_j$ be the result of replacing each polynomial equality $p_i(\vec x,z,\vec y)=0$ with $\bigwedge_{j\leq d_i}p_{i,j}(\vec x,\vec y)=0$. Note that the resulting formula no longer depends on $z$. Then the final disjunct is equivalent to
  \[\exists z\,\bigwedge_{i\leq m}\forall^U\vec y\,\neg\theta_i(\vec x,\vec z,\vec y)\wedge \bigoplus_j\eta^{\pm*}_j\rightarrow\psi^*_j.\]
Because $U$ is small, the left conjunct is provable in $\ACFm$---there always exists an element not algebraic in $G\cup\{\vec x\}$ \cite{MR2235481}[Lemma 2.2(2)]. Therefore the final disjunct is equivalent to this controlled formula.
\end{proof}

\begin{cor}[\cite{MR2235481}[Theorem 3.8]]
  Each $\mathcal{L}(U,\Sigma)$-formula $\phi(\vec x)$ is equivalent in $\ACFm$ to a Boolean combination of formulas of the form
  \[\exists^U\vec y(\theta(\vec y)\wedge\psi(\vec x,\vec y))\]
  where $\theta$ is a $U$-formula and $\psi$ is a $\mathcal{L}$-formula.
\end{cor}
\begin{proof}
  By the previous theorem, $\phi(\vec x)$ is equivalent to a controlled formula. The result follows by writing the components of a controlled formula in disjunctive or conjunctive normal form and distributing quantifiers.
\end{proof}

\section{Real Closed Fields with Small Subgroups}

\subsection{Setting Up}

Van den Dries and G\"unaydin consider a related extension of real closed fields.  Informally, a multiplicative subgroup $\Gamma\subseteq K^\times$ has the Mann property if it has only a small number of additive relationships. We would like to say that every linear equation with rational coefficients,
\[\sum_i q_ix_i=1,\]
has only finitely many solutions where the $x_i$ come from $\Gamma$. This would be slightly too strong: if there happens to be a solution in $\Gamma$ to the homogeneous equation $\sum_{i\in I}q_ix_i=0$ for some non-empty subset $I$ then any solution to $\sum_i q_ix_i=1$ in $\Gamma$ gives rise to infinitely many solutions for the trivial reason that we could always multiply the elements in $I$ by any element of $\Gamma$ to get a new solution.

So to state the Mann property correctly, we need to define this case away.
\begin{definition}
  A solution $\{\gamma_i\}$ to $\sum_i q_ix_i=1$ is \emph{degenerate} if there is some non-empty $I$ so that $\sum_i q_i\gamma_i=0$.

  $\Gamma$ satisfies the Mann property if every linear equation $\sum_i q_ix_i$ has only finitely many \emph{non-degenerate} solutions in $\Gamma$.
\end{definition}
When $K$ is $\mathbb{R}$, typical examples of subgroups with the Mann property $2^{\mathbb{Z}}$, $2^{\mathbb{Q}}$, and $2^{\mathbb{Z}}3^{\mathbb{Z}}$.

When $\Gamma\subseteq\mathbb{R}^{>0}$ is not dense, it is cyclic, and quantifier elimination in this case was shown in \cite{MR808687}, and an algorithm for this quantifier elimination in the case of $2^{\mathbb{Z}}$ was given in  \cite{MR2289704}.

So for the remainder of this section, we focus on the dense case. Given a dense $\Gamma\subseteq\mathbb{R}^{>0}$ satisfying the Mann property, Van den Dries and G\"unaydin define a theory $\RCFU$. This theory is in the language of ordered rings extended by a unary relation $U$, new constant symbols $\{c_\gamma\}_{\gamma\in \Gamma}$, and, for each sequence $k_1,\ldots,k_n$ of integers, an $n$-ary relation symbol $\Sigma_{k_1,\ldots,k_n}$. The theory $\RCFU$ then consists of;
\begin{itemize}
\item the axioms of $\RCF$,
\item axioms stating that $U$ is a dense subgroup of the positive elements,
\item axioms stating that the map $\gamma\mapsto c_\gamma$ is an order-preserving homomorphism,
\item axioms stating, for any integers $k_1,\ldots,k_n$ and $\gamma_1,\ldots,\gamma_n\in\Gamma$, that $\sum_i k_ic_{\gamma_i}>0$ holds exactly when $\sum_i k_i\gamma_i>0$ and that $\Sigma_{k_1,\ldots,k_n}c_{\gamma_1}\cdots c_{\gamma_n}$ holds exactly when $\sum_i k_i\gamma_i=0$,
\item the Mann axioms stating that, for each equation $\sum_i a_ix_i=a_0$ with the $a_i\in\mathbb{Z}^\times$, the finitely many non-degenerate solutions to this equation present in $\Gamma$ are the only ones.
\end{itemize}
This theory is not complete, since there is a great deal left to specify about $\Gamma$, but, as in the algebraic case, Van den Dries and G\"unaydin show a relative quantifier elimination, showing that all formulas are equivalent to Boolean combinations of a simpler family of formulas.

We reuse the names $U$-formula and $KG$-formula in this section with slightly different meanings, analogous but not identical to those in the previous section.
\begin{definition}
  A \emph{$U$-formula} is defined inductively by:
  \begin{itemize}
  \item atomic formulas whose terms are built from variables, $\{c_\gamma\}$, and $\cdot$ (but not $0$, $+$, or $-$) are $U$-formulas,
  \item if $\theta$ is a $U$-formula, so is $\neg\theta$,
  \item if $\theta,\theta'$ are $U$-formulas, so are $\theta\wedge\theta'$ and $\theta\vee\theta'$,
  \item if $\theta$ is a $U$-formula, so are $\exists^Ux\,\theta$ and $\forall^Ux\,\theta$.
  \end{itemize}


  A \emph{KG-formula} is a formula of the form $\exists^U\vec y\,(\theta(\vec y)\wedge\psi(\vec x,\vec y))$ where $\theta$ is a $U$-formula and $\psi$ is a quantifier-free formula in $\mathcal{L}_<(\{c_\gamma\}_{\gamma\in\Gamma})$ (the language of ordered fields together with the constants $\{c_\gamma\}_{\gamma\in\Gamma}$).

\end{definition}

Van den Dries and G\"unaydin's version of quantifier elimination is the following.
\begin{theorem}[\cite{MR2235481}[Lemma 7.6]]
Every formula is equivalent in $\RCFU$ to a Boolean combination of KG-formulas.  
\end{theorem}

Their proof involves beginning with sufficiently saturated models and (given some assumptions ensuring the models are elementarily equivalent) constructing a system of back-and-forth models.

Translating this proof of quantifier elimination into an algorithm is complicated by the fact that the argument doesn't use arbitrary models from their back-and-forth system. The argument goes roughly as follows. Suppose we have two sufficiently saturated models $\mathcal{M}$ and $\mathcal{M}'$, and elements $\vec x\in|\mathcal{M}|^{|\vec x|}$, $\vec x'\in|\mathcal{M}|^{|\vec x'|}$ so that $\vec x$ and $\vec x'$ satisfy the same KG-formulas; we would like to show that $\phi(\vec x)$ holds (in $\mathcal{M}$) exactly when $\phi(\vec x')$ holds (in $\mathcal{M}'$).

The first step is to find an isomorphic pair of small models $M_0\subseteq\mathcal{M}, M'_0\subseteq\mathcal{M}'$ containing $\vec x$ and $\vec x'$, respectively. To construct these, we take $M_0$ to be a substructure of $\mathcal{M}$ elementary with respect to KG-formulas, and similarly for $M_0'$, using the fact that $\vec x$ and $\vec x'$ satisfy the same KG-formulas to ensure they're isomorphic.

Once we have constructed an isomorphic pair $M_i,M'_i$, we need to show that for any $a\in\mathcal{M}$, there is an $M_{i+1}\supseteq M_i\cup\{a\}$ and an $M'_{i+1}\supseteq M'_i$ so that $M_{i+1}$ and $M'_{i+1}$ are isomorphic (and the isomorphism extends the isomorphism between $M_i$ and $M'_i$). We do this by considering extensions of the form $(M_i(z))^{rc}$. (In the case where $z$ is not in $U^{\mathcal{M}}$, we require some further restrictions on $z$.)

To turn this into an explicit algorithm, we need to show that statements about $(M_i(z))^{rc}$ are equivalent to simpler statements---roughly speaking, combinations of statements about $M_i$ and statements about $z$ which do not quantify over $(M_i(z))^{rc}$. A statement about $(M_i(z))^{rc}$ could be a formula $\phi(\vec x,z)$; more generally, it might be a formula $\phi(\vec x,z,\vec w)$, where the $\vec w$ are other elements of $(M_i(z))^{rc}$---for instance, roots of polynomials.

In our algorithm below, this will turn into several steps, during which we progressively reduce statements about $(M_i(z))^{rc}$ to simpler statements.

Our notion of a controlled formula is slightly simpler than in the previous section. It turns out we need to worry less about the precise division of variables.

\begin{definition}
  A \emph{controlled formula} is a formula of the form
   \[\bigoplus\eta^\pm_j(\vec x,\vec y_j)\rightarrow\psi_j(\vec x,\vec y_j)\]
   where the $\eta^{\pm}_j$ are a Boolean combinations of atomic formulas in $\mathcal{L}_<(\{c_\gamma\}_{\gamma\in\Gamma})$ and each $\psi_j$ is rigid under $\eta^{\pm}_j$.
 \end{definition}
 Note that we have not yet placed \emph{any} restriction on the form of the $\psi_j$, though we will always be interested in the case where there is some sort of restriction on them.

\subsection{The Mann Property}

Before proving any quantifier elimination results, we need to get a more general version of the Mann property (essentially the work of Section 5 of \cite{MR2235481}).

\begin{definition}
  Let $t_0,t_1,\ldots, t_n$ be a sequence of terms (in some variables $\vec x$).  We write $\mathrm{Mann}(\{t_i\}_{i\leq n},\{w_{ik}\}_{i\leq n,k\leq b})$ for the formula
  \begin{align*}
    t_0=0\vee&(\forall^U y_1,\ldots,y_n\,\sum_{i\leq n}t_iy_i\neq t_0)\\
    \vee&\bigl(\bigwedge_{k\leq b}\sum_{i\leq n}t_iw_{ik}=t_0\\
    \wedge&   \forall^U y_1,\ldots,y_n\,\left(\sum_{i\leq n}t_iy_i=t_0\wedge\bigwedge_{S\subseteq[1,n],S\neq\emptyset}\sum_{i\in S}t_iy_i\neq 0\right)\rightarrow\bigvee_{k\leq b}\bigwedge_{i\leq n}y_i=w_{ik}\bigr).
  \end{align*}
\end{definition}
That is, if $t_0\neq 0$ and there are any solutions to the equation $\sum_{i\leq n}t_iy_i=t_0$ satisfying $U$ then the $\{w_{ik}\}$ are the complete list of non-degenerate solutions in $U$ to this equation. Note that we are allowed to repeat solutions, so the value $b$ is an upper bound on the number of solutions, not an exact value.

The Mann axioms guarantee that we have $\exists^U\{w_{ik}\}_{i\leq n,k\leq b}\,\mathrm{Mann}(\{t_i\},\{w_{ik}\})$ when the $t_i$ are integers. The next lemma extends this to the case where the $t_i$ are in the subfield generated by $U$.
\begin{lemma}[See \cite{MR2235481}[Lemma 5.5]]
  Let $t_0,t_1,\ldots,t_n$ be terms in the variables $\vec y$. Then there is a $b$ so that $\RCFU$ proves
\[\forall^U\vec y\,\exists^U\{w_{ik}\}_{i\leq n,k\leq b}\ \mathrm{Mann}(\{t_i\},\{w_{ik}\}_{i\leq n,k\leq b}).\]
\end{lemma}
\begin{proof}
  When $t_0=0$, this is immediate, so we assume not. We may expand each $t_i$ as a sum of monomials $\sum_j s_jm_j$ where $s_j$ is a (possibly empty) product of the variables from $\vec y$ and $m_j$ is an integer. Expanding each $t_i$ this way, reindexing, and moving all but one term from the right two the left, a solution to the equation $\sum_i t_i v_i=t_0$ is equivalent to a solution to the equation $\sum_{i\leq n'}m_is_iv'_i=m_0$ together with some auxiliary equations of the form $v'_i=v'_{i'}$ or $v'_i=1$ (to account for the $v'_i$ coming from expansions of the same term or from the right hand side), where the $m_i$ are integers, the $s_i$ are products of the $\vec y$, and $m_0$ is non-zero.
  
For $i'\leq n'$, let us set $\pi(i')=i$ if $v'_{i'}$ in $\sum_{i\leq n'}m_is_iv'_i=m_0$ corresponds to $v_i$ in $\sum_{i\leq n}t_iv_i=t_0$; we also set $\pi(i')=0$ if our auxiliary constraint sets $v'_{i'}$ to $1$ (that is, $v'_{i'}$ originally came from a monomial in $t_0$).

  Consider this related equation $\sum_{i\leq n'}m_iv''_i=m_0$. We already have 
  \[\exists^U\{w_{ik}\}_{i\leq n',k\leq b_0}\,\mathrm{Mann}(\{m_i\},\{w_{ik}\})\]
   for this equation. However a solution to this equation might be degenerate even when the corresponding solution to $\sum_{i\leq n}t_iv_i=t_0$ is not, so we consider something slightly more general.

  We can think of a general, possibly degenerate, solution to $\sum_{i\leq n'}m_iv''_i=m_0$ as being given by a partition $\{1,2,\ldots,n'\}=\bigcup_{j\leq d}I_j$, a non-degenerate solution $\{r'_i\}_{i\in I_0}$ to $\sum_{i\in I_0}m_iv''_i=m_0$, and, for each $0<j\leq d$, a non-degenerate solution to $\{r'_i\}_{i\in I_j}$ to $\sum_{i\in I_j}m_iv''_i=0$.

  For any non-empty subset $I\subseteq\{1,\ldots,n'\}$, fix some $i_0\in I$. Then a non-degenerate solution to $\sum_{i\in I_j}m_iv''_i=0$ is given by a multiple of a non-degenerate solution to $\sum_{i\in I\setminus\{i_0\}}m_iv'''_i=-m_{i_0}$.

  So, applying the Mann axioms to all equations $\sum_{i\in I}m_iv''_i=m_0$ and $\sum_{i\in I\setminus\{i_0\}}m_iv'''_i=-m_{i_0}$ simultaneously, we obtain a finite list of all possible combinations consisting of a partition $\{I_j\}$ and non-degenerate solutions $\{w^0_{i}\}_{i\in I_0},\{w^j_{i}\}_{i\in I_j\setminus\{i^j_0\}}$.  We need to describe how to obtain a list of non-degenerate solutions to our original equation $\sum_{i\leq n}t_iv_i=t_0$ from this, and then show that this list contains all possible non-degenerate solutions.

  So let a partition $\{I_j\}$ and non-degenerate solutions $\{w^0_{i}\}_{i\in I_0},\{w^j_{i}\}_{i\in I_j\setminus\{i^j_0\}}$ be given. For notational convenience, for $0<j$ set $w_{i^j_0}=1$, so we may then view $\{w^j_{i}\}_{i\in I_j}$ as a solution to $\sum_{i\in I_j}m_iv'''_i=0$. Note that if we multiply this solution by any fixed value, we obtain another solution to the same equation.

  We now attempt to read a solution to $\sum_{i\leq n'}m_is_iv'_i=0$ satisfying the auxiliary constraints off of these values.  First, for $i\in I_0$, we must take $v'_i=w^0_{i}/s_i$, and this gives a solution to $\sum_{i\in I_0}m_is_iv'_i=m_0$. For $0<j$ and $i\in I_j$, we are going to take $v'_i=c_jw^j_{i}/s_i$ where $c_j$ is shared by all $i\in I_j$; we need to identify the correct values of $c_j$ and see if they can be chosen in a consistent way.

  For any $j>0$, if there any $i^*\in I_j$ with $\pi(i^*)=0$, we are only interested in solutions with $v'_{i^*}=1$, so we must have $c_j=s_{i^*}/w^j_{i^*}$. If we have already determined $c_j$ and there is an $i^+\in I_j$ and an $i^*\in I_{j'}$ with $\pi(i^*)=\pi(i^+)$ then we need to have $v'_{i^*}=v'_{i^+}$. This means we must have $c_{j'}w^{j'}_{i^*}/s_{i^*}=c_jw^j_{i^+}/s_{i^+}$, so we must take $c_{j'}=c_jw^j_{i^+}s_{i^*}/s_{i^+}w^{j'}_{i^*}$. Applying this process iteratively to determine the values of the $c_j$, there are three possibilities: we eventually determine all the values $c_j$, we eventually reach two conflicting values for the same $c_j$, or we fail to determine all the values $c_j$.

  We apply this process to all (finitely many) choices of a partition $\{I_j\}$ and non-degenerate solutions $\{w^0_{i}\}_{i\in I_0},\{w^j_{i}\}_{i\in I_j\setminus\{i^j_0\}}$. Our list of solutions to $\sum_{i\leq n'}m_is_iv'_i=m_0$ (and therefore to $\sum_{i\leq n}t_iv_i=t_0$) consists of all solutions in the first case, where we determine all values of the $c_j$ uniquely without reaching a contradiction. In the other cases, $\{I_j\}$, $\{w^0_{i}\}_{i\in I_0},\{w^j_{i}\}_{i\in I_j\setminus\{i^j_0\}}$ does not give rise to a non-degenerate solution to $\sum_{i\leq n}t_iv_i=t_0$. (As we will see, the case where the requirements conflicts corresponds to a non-solution and the cases where they do not uniquely determine the $c_j$ corresponds to a degenerate solution).
  
  What remains is to show that we have obtained all non-degenerate solutions to $\sum_{i\leq n}t_iv_i=t_0$. To see this, consider any non-degenerate solution $\sum_{i\leq n}t_iu_i=t_0$. This leads to a solution $\sum_{i'\leq n'}m_iu''_i=m_0$ by taking $u''_i=s'_iu_{\pi(i)}$. We may partition $\{1,\ldots,n'\}$ into minimal sets $I_j$ with $\sum_{i\in I_j}t'_iu''_i=0$, and let $I_0$ be the remaining indices. Then our list of partitions and solutions must include this choice of partition with the solutions $u''_i$ (for $j=0$) and $u''_i/u''_{i^j_0}$ (for $j>0$).

  When we apply the process above, we wish to show that we recover this solution. Certainly we do not run into a conflicting definition of $c_j$ so we only need to check that each $c_j$ eventually gets defined. Let $J_0$ be the set of $j>0$ so that $c_j$ does not get defined and let $I=\bigcup_{j\in J_0}I_j$. Then for each $i\leq n$, we must have that $\pi^{-1}(i)$ is either contained in $I$ or disjoint from $I$. Since $\sum_{i\in I_j}m_iu''_i=0$ for each $j\in J_0$, also $0=\sum_{i\in I}m_iu''_i=\sum_{i\in I}s_im_iu_i$. This would mean that $\sum_{i\mid \pi^{-1}(i)\subseteq I}t_iu_i=0$, and therefore this solution was degenerate.
\end{proof}




The next lemma extends this to show $\exists^U\{w_{ik}\}_{i\leq n,k\leq b}\,\mathrm{Mann}(\{t_i\},\{w_{ik}\})$ for arbitrary terms $t_i$. We will need to work with names for elements of $\mathbb{Q}(U)$. We cannot name these elements uniformly, so we introduce a scheme for describing bounded portions of this field. We write $\mathbb{Q}_{\leq d}(U)$ for those elements which can be written in the form $\frac{\sum_{i\leq b}u_i}{\sum_{i\leq b'}u'_i}$ where $b,b'\leq d$ and each $u_i,u'_i$ is in $U\cup -U$.

Note that if $a_1,\ldots,a_n\in\mathbb{Q}_{\leq d}(U)$ then $\sum_{i\leq n}a_i$ is in $\mathbb{Q}_{\leq nd^n}(U)$, and if $a,b\in\mathbb{Q}_{\leq d}(U)$ then $ab\in\mathbb{Q}_{\leq d^2}(U)$.

\begin{lemma}\label{thm:general_mann}[See \cite{MR2235481}[Proposition 5.6]]
  Let $\vec x$ be variables and let $t_0,\ldots,t_n$ be terms in $\vec x$. Then 
  $\RCFU$ proves a controlled formula
  \[\forall\vec x\ \bigotimes_{j\in J}\left(\eta^{\pm}_j(\vec x,\vec y_j)\rightarrow \mathrm{Mann}(\{t_i\},\{w^j_{ik}\})\right)\]
  for some terms $w^j_{ik}$.
 \end{lemma}
\begin{proof}


  We will have an element $-1$ in $J$ and take $\eta^+_{-1}$ to be $t_0=0$, in which case $\mathrm{Mann}(\{t_i\},\cdot)$ is trivial, so we focus on the cases where $t_0\neq 0$.

  Consider a subset $S\subseteq\{0,1,\ldots,n\}$ containing $0$. For $0\leq i\leq n$, let $t'_i$ be $t_i/t_0$.  We wish to express that $\{t'_i\}_{i\in S}$ forms a basis for the space generated by the $t'_i$ over $\mathbb{Q}(U)$. We cannot write this down as a single formula in general, because elements of $\mathbb{Q}(U)$ are arbitrary sums of elements of $U$.

Instead, for any bounds $d,d'$, there is a formula $\theta_{S,d,d'}$ expressing that $\{t'_i\}_{i\in S}$ is nearly a basis, in the sense that:
  \begin{itemize}
  \item each $t'_j$ with $j\not\in S$ is a linear combination of $\{t'_i\}_{i\in S}$ with coefficients from $\mathbb{Q}_{\leq d}(U)$, and
  \item the equation $\sum_{i\in S}t'_iy_i=0$ has no solutions with coefficients from $\mathbb{Q}_{\leq d'}(U)$.
  \end{itemize}

  Fix the sequence $d_{n+1}=1$, $d_{i-1}=2n^2d_{i}^{2n}$. (The reasons for this choice will become apparent later.) We will define bounds $r_S$ for each $S$ later, and take the set $J$ to contain elements which are pairs $(S,m)$ where $S\subseteq\{0,1,\ldots,n\}$ with $0\in S$ and $m\leq r_S$. Fix an ordering of the sets $S$. For every $S$, $\exists^U\vec y\,\eta^+_{S,m}\wedge\forall^U\vec y\,\eta^-_{S,m}$ will imply that $S$ is least in this ordering so that $\theta_{S,d_{|S|},d_{|S|-1}}$ holds.

  We now work in some case $(S,m)$, so $\theta_{S,d_{|S|},d_{|S|-1}}$ holds and there are terms $a_{ij}$, representing elements of $\mathbb{Q}_{\leq d_{|S|}}(U)$, which can be expressed using the $\vec j$ so that, for all $i\not\in S$ and $j\in S$, $t'_i=\sum_{j\in S}a_{ij}t'_j$. For $S=\{0\}$, we will have $r_S=0$, so the only value of $m$ will be $0$, and we will take $\eta^+_{\{0\},0},\eta^-_{\{0\},0}$ to say that not only is $\{0\}$ least so that $\theta_{S,d_1,d_0}$ holds, but also that $\mathrm{Mann}(\{a_{i0}\}, \{w_{ik}\}_{i\leq n,k\leq b})$ with $b$ given by the previous lemma. We then take take $w^{\{0\},0}_{ik}=w_{ik}$.

Suppose $|S|>1$. In addition to the $a_{ij}$ fixed above with $i\not\in S$, for $i\in S$ and $j\in S$, we take $a_{ij}$ to be $1$ if $i=j$ and $0$ otherwise. This means that we have $t'_i=\sum_{j\in S}a_{ij}t'_j$ for all $i\in\{0,1,\ldots,n\}$.

  We next show that a solution to $\sum_{i\leq n}t'_iv_i=1$ in $U$ is equivalent to a simultaneous solution to $\sum_{i\leq n}a_{i0}v_i=1$ and, for $j\in S$, $\sum_{i\leq n}a_{ij}v_i=0$.

  The right to left direction is immediate: if we assume $\sum_{i\leq n}a_{i0}v_i=1$ and, for $j\in S$,  $\sum_{i\leq n}a_{ij}v_i=0$ then we have
  \[\sum_{i\leq n}t'_iv_i=\sum_{i\leq n}\sum_{j\in S}a_{ij}t'_jv_i=\sum_{j\in S}\sum_{i\leq n}a_{ij}t'_jv_j=1.\]

  For the left to right direction, suppose $\sum_{i\leq n}t'_iv_i=1$. Then, replacing $t'_i$ by $\sum_{j\in S}a_{ij}t'_j$, we have $\sum_{j\in S}\sum_{i\leq n} a_{ij}v_i=1$. Suppose for a contradiction that, for some $j_0\in S\setminus\{0\}$, $\sum_{i\leq n}a_{ij_0}v_i\neq 0$. Then we have
  \[t'_{j_0}=1-\sum_{j\in S\setminus\{j_0\}}t'_j\frac{\sum_{i\leq n} a_{ij}v_i}{\sum_{i\leq n}a_{ij_0}v_i}.\]
  Taking $b_0=1-\frac{\sum_{i\leq n} a_{i0}v_i}{\sum_{i\leq n}a_{ij_0}v_i}$ and, for $j\in S\setminus\{0,j_0\}$, $b_j=\frac{\sum_{i\leq n} a_{ij}v_i}{\sum_{i\leq n}a_{ij_0}v_i}$, observe that $b_0$ and $b_j$ belong to $\mathbb{Q}_{\leq 2n^2d_{|S|}^{2n}}(U)=\mathbb{Q}_{\leq d_{|S|-1}}(U)$, which contradicts $\theta_{S,d_{|S|},d_{|S|-1}}$.

  Therefore $\sum_{i\leq n}a_{ij}v_i= 0$ for each $j\in S\setminus\{0\}$, and therefore $\sum_{i\leq n}a_{i0}v_i=1$.

  We let $r_S$ be the $b$ given by the previous lemma for $\sum_{i\leq n}a_{i0}v_i=1$. We take $\eta^+_{S,m},\eta^-_{S,m}$ to say that, in addition to $S$ being least so that $\theta_{S,d_{|S|},d_{|S|-1}}$ holds, that additionally $\mathrm{Mann}(\{a_{i0}\},\{w_{ik}\}_{i\leq n,k\leq b})$ holds, and that, for any $k\leq b$, we have $\sum_{i\leq n}a_{ij}y_{ik}=0$ for all $j$ if and only if $k\leq m$. Then we take $w^{S,m}_{ik}=w_{ik}$ for $k\leq m$, so $\mathrm{Mann}(\{t_i\},\{w^{S,m}_{ik}\}_{k\leq m})$ holds.

\end{proof}

The stratification of $\mathbb{Q}(U)$ is a pattern that occurs often in proof mining. Implicit in the preceding lemma is the idea that the quantitative version of ``a basis for $S$ over $\mathbb{Q}(U)$ exists'' is ``for any function $F:\mathbb{N}\rightarrow\mathbb{N}$ there is a $d$ (depending on $F$ and $|S|$) so that there are linear combinations $\{t'_i\}$ of elements of $S$ with coefficients from $\mathbb{Q}_{\leq d}(U)$ so that $\sum_i t'_iy_i=0$ has no solutions with coefficients from $\mathbb{Q}_{\leq F(d)}(U)$''. Then any particular quantitative use of a basis requires this for a specific $F$ which depends on how the basis gets used. This pattern is sometimes called \emph{metastability} \cite{MR3111912,MR2550151,iovino_duenez,MR3141811,tao08Norm}, where the particular bound $F$ on the fragment depends on the specific way the basis will be used.

\subsection{Quantifier Elimination}

The first thing we need to do is show that if an element $u$ from $U$ is algebraically dependent over $\vec x$ then it is already multiplicatively dependent. The argument depends crucially on which variables from $\vec x$ are in $U$ and which are not, so the first lemma assumes we are given, in the form of the set $D$, which elements are supposed to be in $U$. In the next lemma, we will take a disjunction over all choices of $D$.


\begin{lemma}[See \cite{MR2235481}, Lemma 5.12]
  Let $\vec x$ be variables, let $D\subseteq\{\vec x\}$, and consider an equation
\[\sum_{r\leq d}\sum_{i\leq n_r} s_{r,i}t_{r,i}u^r=0\]
  where:
  \begin{itemize}
  \item each $s_{i,r}$ is a monomial containing only variables in $D$, and
  \item each $t_{i,r}$ is a monomial containing only variables not in $D$.
  \end{itemize}
Then $\RCFU$ proves a controlled formula
 \[\forall\vec x\ \bigotimes_{j\in J}\left(\eta^{\pm}_j(\vec x,\vec y^j)\rightarrow\bigwedge_{v\in D}\neg Uv\wedge\bigwedge_{v\not\in D}Uv\rightarrow(\forall^Uu\,\sum_{r\leq d}\sum_{i\leq n_r}s_{r,i}t_{r,i}u^r=0\leftrightarrow\psi_j)\right)\]
 where each $\psi_j$ is a positive Boolean combination of equalities and inequalities in which the only atomic formulas involving $u$ have the form $tu^m=s$ where $s,t$ are monomials.


\end{lemma}
\begin{proof}
The idea is this: if we have a potential solution to the equation $\sum_{r,i}s_{r,i}t_{r,i}u^r=0$, we would hope to take some $r_0,i_0$ so that $s_{r_0,i_0}t_{r_0,i_0}\neq 0$ and with $r_0$ minimal so this happens, and then we could rearrange our equation to $\sum_{r,i\neq r_0,i_0}s_{r,i}\frac{t_{r,i}}{t_{r_0,i_0}}u^{r-r_0}=-s_{r_0,i_0}$. Then the part of this in $U$, 
$\frac{t_{r,i}}{t_{r_0,i_0}}u^{r-r_0}$, is a solution to the equation $\sum_{r,i\neq r_0,i_0}s_{r,i}v_{i,r}=-s_{r_0,i_0}$. Since the Mann property gives us a finite list of non-degenerate solutions to this equation, we would hope to choose some other $r,i$ with $t_{r,i}$ non-zero and then have $\frac{t_{r,i}}{t_{r_0,i_0}}u^{r-r_0}=w_{r,i,k}$ for one of our finite list of possible solutions $w_{r,i,k}$.

The problem is that we might be dealing with a degenerate solution. With a degenerate solution, though, there should be some smaller set of pairs $I$ so that $\sum_{(r,i)\in I}s_{r,i}t_{r,i}u^r=0$ as well, so we could instead apply the same idea to this smaller equation. To complete this argument, it will be convenient to deal with all possible sets of pairs at once.
  
  Let $I^*$ be the collection of pairs $(r,i)$ with $r\leq d$, $i\leq n_r$. For each non-empty $I\subseteq I^*$, we can consider the equation $\sum_{(r,i)\in I}s_{r,i}t_{r,i}u^r=0$. Let $r_I$ be least so that there is an $i\leq n_{r_I}$ with $(r_I,i)\in I$, and fix some such value $i_I$. We can then consider the equation $\sum_{(r,i)\in I\setminus\{(r_I,i_I)\}}s_{r,i}v_{r,i}=-s_{r_I,i_I}$. We may apply Lemma \ref{thm:general_mann} to this equation, obtaining cases indexed by $J_I$ and a finite list of possible solutions for each $j\in J_I$.

  Let $J$ be the set of functions $f$ whose domain is non-empty subsets of $I^*$ and so that, for each $I$, $f(I)\in J_I$. We take $\eta^+_f$ to be $\bigwedge_I\eta^+_{f(I)}$ and $\eta^-_f$ to be $\bigwedge_I\eta^-_{f(I)}$.

  We now need to define the formulas $\psi_f$. We will take $\psi_f$ to be the formula stating that  there is a partition $I^*=\bigcup_{p<m}I_p$ into non-empty sets such that, for each $p<m$, either:
  \begin{itemize}
  \item $\sum_{(r,i)\in I_p}s_{r,i}t_{r,i}=0$, or
  \item one of the solutions $\{w_{r,i,k}\}_{(r,i)\in I_p\setminus\{(r_{I_p},i_{I_p})\}}$ to the equation
    \[\sum_{(r,i)\in I_p\setminus\{(r_{I_p},i_{I_p})\}}s_{r,i}v_{r,i}=-s_{r_{I_p},i_{I_p}}\]
    has the property that $t_{r,i}u^{r-r_{I_p}}=a_{r,i,k}t_{r_{I_p},i_{I_p}}$ for all $(r,i)\in I_p\setminus\{(r_{I_p},i_{I_p})\}$.
  \end{itemize}
  This can be expressed by a formula of the desired kind: a disjunction, over all possible partitions and assignments of partitions to the two cases, of a Boolean combination of equalities and the only equalities containing $u$ appear positively and have the specified form; replacing negated equalities with a disjunction of inequalities gives the formula we need, so it suffices to establish that this equivalence holds.

  We claim that, under the appropriate assumptions, $\sum_{r,i}s_{r,i}t_{r,i}u^r=0$ holds if and only if $\psi_f$ holds. For the left-to-right direction, suppose that $\sum_{r,i}s_{r,i}t_{r,i}u^r=0$. We can successively choose pairwise disjoint minimal non-empty sets $I_p$ such that $\sum_{(r,i)\in I_j}s_{r,i}t_{r,i}u^r=0$. For each such $I_p$, either $\sum_{(r,i)\in I_p}s_{r,i}t_{r,i}=0$ or $s_{r_{I_p},i_{I_p}}\neq 0$ (because $I_p$ is minimal). In the latter case, $\{\frac{t_{r,i}u^{r-r_{I_p}}}{t_{r_{I_p},i_{I_p}}}\}$ is a non-degenerate solution to $\sum_{(r,i)\in I_p\setminus\{(r_{I_p},i_{I_p})\}}s_{r,i}v_{r,i}=-s_{r_{I_p},i_{I_p}}$ (because the minimality of $I_p$ ensures non-degeneracy). Therefore $t_{r,I}u^{r-r_{I_p}}=w_{r,i,k}t_{r_{I_p},i_{I_p}}$ for one of the solutions $\{w_{r,i,k}\}$ to this equation.

For the right-to-left direction, suppose we are given $I=\bigcup_{p<m}I_p$ and, for each $p<m$ with $\sum_{(r,i)\in I_p}s_{r,i}t_{r,i}\neq 0$, solutions $\{w_{r,i,k}\}$ so that $\sum_{(r,i)\in I_p\setminus\{(r_{I_p},i_{I_p})\}}s_{r,i}w_{r,i,k}=-s_{r_{I_p},i_{I_p}}$ and $t_{r,i}u^{r-r_{I_p}}=w_{r,i,k}t_{r_{I_p},i_{I_p}}$. Let $P$ be the set of $p$ with $\sum_{(r,i)\in I_p}s_{r,i}t_{r,i}\neq 0$. Then we have
 \begin{align*}
   \sum_{(r,i)}s_{r,i}t_{r,i}u^r
   &=\sum_{p}\sum_{(r,i)\in I_p}s_{r,i}t_{r,i}u^r\\
   &=\sum_{p\in P} s_{r_{I_p},i_{I_p}}t_{r_{I_p},i_{I_p}}u^{r_{I_p}}+\sum_{(r,i)\in I_p\setminus\{(r_{I_p},i_{i_p})\}}s_{r,i}t_{r,i}u^{r_{I_p}}\frac{w_{r,i,k}t_{r_{I_p},i_{I_p}}}{t_{r,i}}\\
   &=\sum_{p\in P} s_{r_{I_p},i_{I_p}}t_{r_{I_p},i_{I_p}}u^{r_{I_p}}+u^{r_{I_p}}t_{r_{I_p},i_{I_p}}\sum_{(r,i)\in I_p\setminus\{(r_{I_p},i_{i_p})\}}s_{r,i}w_{r,i,k}\\
   &=\sum_{p\in P} s_{r_{I_p},i_{I_p}}t_{r_{I_p},i_{I_p}}u^{r_{I_p}}+u^{r_{I_p}}t_{r_{I_p},i_{I_p}}(-s_{r_{I_p},i_{I_p}})\\
   &=0.
 \end{align*}
\end{proof}


\begin{lemma}
  Let $\vec x$ variables and let $t_i$ be terms in $\vec x$. Then $\RCFU$ proves a controlled formula
  \[\forall \vec x(\bigoplus_{j\in J}\eta^{\pm}_j(\vec x,\vec y^j)\rightarrow\forall^Uu\,\left(\sum_i t_iu^i=0\leftrightarrow\psi_j(\vec x,\vec y^j,w)\right))\]
  where each $\psi_j$ is a positive Boolean combination of atomic formulas and $\neg Ut$ in which the only atomic formulas involving $u$ have the form $tu^m=s$ where $s,t$ are monomials and each formula $Ut$ or $\neg Ut$ has $t$ a single variable.

\end{lemma}
\begin{proof}
  For each $D\subseteq\{\vec x\}$, we may expand  the polynomial $\sum_i t_iu^i$ into a sum of monomials and then, based on the choice of $D$, split each monomial into a product of powers of variables not in $U$ and those in $U$---that is, as in the previous lemma, write it in the form $\sum_r\sum_i s_{r,i}t_{r,i}u^r$ where $s_{r,i}$ is a product of variables not in $U$ and $t_{r,i}$ is a product of variables in $U$. We may apply the previous lemma, obtaining $J_D,\eta^{\pm}_{D,j},\psi_{D,j}$. Let $J$ consist of all functions whose domain is the subsets of $\{\vec x\}$ and with $f(D)\in J_D$ for all $D$. Define $\eta^{\pm}_f$ to be $\bigwedge_D\eta^{\pm}_{D,f(D)}$ for all $D$ and $\psi_f$ to be
  \[\bigvee_D (\bigwedge_{v\in D}\neg Uv\wedge_{v\not\in D}Uv)\wedge\psi_{D,f(D)}.\]




\end{proof}

We can finally eliminate an existential quantifier, albeit in a very restricted setting---a quantifier restricted to $U$ over an $\mathcal{L}_<(\{c_\gamma\}_{\gamma\in\Gamma})$ formula. We will then build on this special case.

\begin{definition}
    A \emph{guarded formula} is a formula of the form $\exists^Uu\,\phi'$ where $\phi'$ is a Boolean combination of equalities and inequalities.
  \end{definition}

In the model theoretic argument, van den Dries and G\"unaydin show that certain isomorphisms between models can be extended. The first case they need to consider is extending a model by an element of $U$. In doing so, they need to keep track of roots: if we add an element $u$ in $U$ whose positive square root is also in $U$, it needs to map to an element whose positive square root is also in $U$, and more generally for other roots, and with multiples of $u$ by other elements of $U$ as well.

To capture this, the formulas we work with through most of our argument will include ``root formulas'', which will express that terms are powers of elements of $U$.

\begin{definition}

  A term $t$ is a \emph{monomial} if it does not contain $+$ (and so it built from variables and constant symbols using $\cdot$).

  A \emph{root formula} is a formula of the form $\exists^Uu\, t_nu^n+t_0=0$ where $t_0,t_n$ are monomials.
\end{definition}
We call this a root formula because it states that the $n$-th root of $-t_0/t_n$ is in $U$. Note that we place the further requirement that $t_0$ and $t_n$ are monomials---typically $t_0,t_n$ will be products of elements known to be in $U$.

\begin{definition}
  A \emph{permitted formula} in $\vec x$ is a Boolean combination of:
  \begin{itemize}
  \item atomic formulas in $\mathcal{L}_<(\{c_\gamma\}_{\gamma\in\Gamma})$,
  \item root formulas.
  \end{itemize}
\end{definition}

\begin{lemma}\label{thm:equiv_guarded}
  Let $\vec x$ variables and let $\phi$ be a guarded formula. Then $\RCFU$ proves a controlled formula
  \[\forall \vec x(\bigoplus_{j\in J}\eta^{\pm}_j(\vec x,\vec y^j)\rightarrow\forall^Uu\,\left(\sum_i t_iu^i=0\leftrightarrow\psi_j(\vec x,\vec y^j,w)\right))\]
  where each $\psi_j$ is a permitted formula.

\end{lemma}
\begin{proof}
  $\phi$ has the form $\exists^Uu\,\phi'$. Replace negated equalities $s\neq t$ in $\phi'$ with $s<t\vee t<s$, and similarly for negated inequalities.   Let $A$ be the set of equalities in $\phi'$. For each equality in $\phi'$, we may apply the previous lemma, obtaining some collection $\{J_a\}_{a\in A}$ and corresponding formulas $\eta^{\pm}_{a,j}, \psi_{a,j}$.

  Let $J$ consist of functions $f$ mapping each $a\in A$ to $f(a)\in J_a$, and for each $f\in J$, let $\eta^{\pm}_f$ be $\bigwedge_{a\in A}\eta^{\pm}_{a,f(a)}$.

  Given an $f$, under the assumptions $\eta^+_f(\vec x,\vec y_f)$ and $\forall^U\vec y_f\,\eta^-_f(\vec x,\vec y_f)$, $\phi'$ is equivalent to a positive Boolean combination of atomic formulas (including $Ut$) and $\neg Ut$ in which additionally:
  \begin{itemize}
  \item the only equalities involving $u$ have the form $tu^m=s$ where $s,t$ are monomials, and
  \item each $Ut$ or $\neg Ut$ has $t$ a single variable.
  \end{itemize}
  We write this formula in disjunctive normal form, and then distribute the existential and remove formulas which do not depend on $u$, so $\exists^Uu\,\phi'$ is equivalent to a formula $\bigvee_i\phi'_i\wedge\exists^Uu\,\phi''_i$ where each $\phi'_i$ and $\phi''_i$ is a conjunction of formulas as above, with $\phi''_i$ containing only atomic formulas containing $u$ and $\phi'_i$ containing all other conjuncts.

  In $\phi'_i$, we may replace $Ut$ with the root formula $\exists^Uv\,t=v$ (and $\neg Ut$ with its negation), so $\phi'_i$ is equivalent to a permitted formula. So it suffices to show that $\exists^Uu\,\phi''_i$ is equivalent to a permitted formula. If $\neg Uu$ appears in $\phi''_i$ then $\exists^Uu\,\phi''_i$ is equivalent to $\bot$, and $Uu$ is redundant, so we may assume $\phi''_i$ contains only equalities and inequalities.

  If $\phi''_i$ contains no equalities then the density of $U$ implies that $\exists^Uu\,\phi''_i$ is equivalent to $\exists u\,0<u\wedge\phi''_i$.

  So suppose there is an equality $tu^m=s$ in $\phi''_i$. Since $Ut$ implies $0<t$ and $s/t$ has a unique positive $m$-th root (if any), $\exists^Uy\,tu^m=s\wedge\phi'''_i$ is equivalent to $(\exists^Uu\,tu^m=s)\wedge\exists u\,(0<u\wedge \phi'''_i)$. Then quantifier elimination for $\RCF$ gives us a quantifier-free equivalent for $\exists u\,\phi'''_i$, giving us a permitted formula.
\end{proof}

As in the previous section, there is a dichotomy between adding algebraic and transcendental elements. The next lemma is the case of a quantifier over $U$ with a transcendental element: it shows that if $\exists^Uu\,\phi(\vec x,u)$ holds then either $u$ satisfies one of a finite list of non-trivial algebraic identities, or a permitted formula holds.

\begin{lemma}\label{thm:R_exists_U_free_elim}[See \cite{MR2235481}, Theorem 7.1, Case 1]
  Let $\vec x$ be variables and let $\phi(\vec x,w)$ be a permitted formula.
Then $\RCFU$ proves a controlled formula

  \[\forall\vec x(\bigoplus_{j\in J}\eta^{\pm}_j(\vec x,\vec y_j)\rightarrow\left((\exists^Uu\, \nu(\vec x,\vec y_j,u)\wedge \phi(\vec x,u))\leftrightarrow\psi_j(\vec x,\vec y_j)\right))\]
  where each $\psi_j$ is a permitted formula and $\nu$ is a conjunction of the form $\bigwedge_{i\leq m}\left(\sum_{i'\leq n_{i}}t_{i,i'}u^{i'}=0\rightarrow \bigwedge_{i'\leq n_{i,i'}}t_{i,i'}\neq 0\right)$.
\end{lemma}
\begin{proof}
  Replacing negated equalities with a conjunction of inequalities, and similarly for negated inequalities, we may assume each atomic formula in $\phi$ appears positively.  We may write each equality in $\phi$ as a polynomial in $u$, $\sum_{i\leq d}t_iu^i=0$. The formula $\nu$ will be the conjunction of $\sum_{i\leq d}t_iu^i=0\rightarrow\bigwedge_{i\leq d}t_i=0$ for each equality in $\phi$, and therefore, under the assumption $\nu$, $\phi$ will be equivalent to a formula where no equality contains $u$.

  We may rearrange $\phi$ to be in disjunctive normal form and distribute the existential over the disjunction, so we may assume $\phi$ has the form $\phi_0(\vec x)\wedge\phi_1(\vec x,u)\wedge\phi_2(\vec x,u)$ where $\phi_0$ is a conjunction of atomic formulas not involving $u$, $\phi_1$ is a conjunction of inequalities, and $\phi_2$ is a conjunction of root formulas and negations of root formulas (since equalities do not depend on $u$).
  
  By the density of $U$, the inequalities $\exists^U u\,\phi_1(\vec x,u)$ are satisfiable in $U$ just if they are satisfiable in general by a positive element, and by quantifier elimination for RCF, $\exists u\,0<u\wedge\phi_1(\vec x,u)$ is equivalent to a quantifier-free formula in the language of ordered fields, $\phi'_1(\vec x)$.
  
  Further, since $U$ is dense, the $n$-th powers of elements of $U$ is also dense for each $n$. In particular, we may choose a bound $N$ sufficiently large and, in any positive interval, find an element $h$ of $U$ which is an $n$-th power in $U$ for each $n\leq N$. Therefore if $\phi'_1(\vec x)\wedge\phi_2(\vec x,u)$ holds then we may choose an $h$ in $U$ so that $\phi_1(\vec x,hu)\wedge\phi_2(\vec x,hu)$ holds.
  
  So it remains to find a quantifier-free equivalent for $\exists^Uu\,\phi_2(\vec x,\vec y,u)$. Each conjunct in $\phi_2$ has the form $\exists^Uw\,t_iw^{m_i}=s_iu^{k_i}$ or $\exists^Uw\,t_iu^{k_i}w^{m_i}=s_i$ or a negation of one of these.   Raising each side of these equations to a suitable power, we may assume that there is a single value $m=m_i$ for each $i$.

  Consider the quotient group $U/U^m$ in some model. Whether an element $u$ satisfies $\phi_2$ is determined by its image in this quotient relative to the images of the $s_i,t_i$.  The images of the various $s_i,t_i$ generate a subgroup of this group, $T$, and a hypothetical solution $u$ would generate an extension $T\otimes\mathbb{Z}_n$ for some $n$ a factor of $m$ (possibly $1$ if the image of $u$ is already in $T$).  So $\exists^Uu\, \phi_2(\vec x,\vec y,u)$ is equivalent to $\bigvee_{i\in I}\phi_2(\vec x,\vec y,r_i)$ if we can ensure that the terms $r_i$ range over elements of $U$ whose image includes every element of every group $T\otimes(\mathbb{Z}/n\mathbb{Z})$ which can be found in $U/U^m$. (Note that it is possible that $\phi_2(\vec x,\vec y,u)$ would be solved by an element with image in $T\oplus(\mathbb{Z}/n\mathbb{Z})$, but that there is no copy of $T\oplus(\mathbb{Z}/n\mathbb{Z})$ in $U/U^m$.)

  So let $M$ be the set of factors of $m$, let $d$ be larger than the maximum possible size of $T$, and let $J$ consist of all functions from $M$ to $\{0,\ldots,d\}$. Given $f\in J$, let $\eta^{\pm}_f$ be the formulas specifying that, for each $n\in M$, the image of $\vec y^f$ in $U/U^m$ includes a subgroup of the form $(\mathbb{Z}/n\mathbb{Z})^{f(n)}$, and if $f(n)<d$ then $U/U^m$ contains no subgroups of the form $(\mathbb{Z}/n\mathbb{Z})^{f(n)+1}$.

  Then we take the $r_i$ to enumerate all products of a power of some $\vec y^f_i$ up to $m$ with a product of powers of the $s_i,t_i$ up to $m$. The $r_i$ certainly enumerate $T$, and if $T\oplus(\mathbb{Z}/n\mathbb{Z})$ is contained in $U/U^m$ then the $r_i$ must contain some element generating a cycle of length $n$ and not contained in $T$, so in this case the $r_i$ enumerate $T\oplus (\mathbb{Z}/n\mathbb{Z})$ as well.  The formula $\bigvee_{i\in I}\phi_2(\vec x,\vec y,r_i)$ is a Boolean combination of formulas $\exists^Uw\,tw^{m}=sr^{k}$, and therefore permitted. 
\end{proof}

\begin{lemma}\label{thm:R_exists_U_bound_elim}
  Let $\vec x$ be variables, let $\phi(\vec x,u)$ be a permitted formula, and let $t_i$, $i\leq d$ be terms. Then $RCFU$ proves controlled formulas 
  \[\forall\vec x\,(\bigoplus_{j\in J}\eta^{\pm}_j(\vec x,\vec y_j)\rightarrow\left([\exists u\,(\sum_{i\leq d}t_iu^i=0\wedge\bigvee_i t_i\neq 0\wedge\phi(\vec x,u)]\leftrightarrow\psi_j(\vec x,\vec y_j)\right)\]
and
\[\forall\vec x\,(\bigoplus_{j\in J}\eta^{\pm}_j(\vec x,\vec y_j)\rightarrow\left([\exists^U u\,(\sum_{i\leq d}t_iu^i=0\wedge\bigvee_i t_i\neq 0\wedge\phi(\vec x,u)]\leftrightarrow\psi'_j(\vec x,\vec y_j)\right)\]
where the $\psi_j,\psi'_j$ are permitted formulas.
\end{lemma}
\begin{proof}
  By quantifier elimination for real closed fields, for each $m\leq d$ there is a quantifier-free formula $\rho_m(u)$ in the language of ordered fields so that $\rho_m(u)$ holds exactly when $u$ is the $m$-th largest root of the polynomial $\sum_i t_iv^i=0$, so $\sum_{i\leq d}t_iu^i=0$ is equivalent to $\bigvee_{m\leq d}\rho_m(u)$. Splitting the existential over the disjunction $\bigvee_{m\leq d}\rho_m(u)$, we may consider a formula $\exists u\,\rho_m(u)\wedge\phi$ (respectively $\exists^U u\,\rho_m(u)\wedge\phi$ for the second part). We may further write $\phi$ in disjunctive normal form and distribute the existential over the disjuncts. So we assume $\phi$ is a conjunction of equalities, inequalities, and root formulas.

  Since $\rho_m(u)$ uniquely identifies $u$, we may treat the conjuncts distinctly---$\exists u\, \rho_m(u)\wedge\bigwedge_{r\leq k}\phi_r$ is equivalent to $\bigwedge_{r\leq k}\exists u\,\rho_m(u)\wedge\phi_r$ and $\exists^U u\, \rho_m(u)\wedge\bigwedge_{r\leq k}\phi_r$ is equivalent to $(\exists^U u\, \rho_m(u))\wedge\bigwedge_{r\leq k}\exists u\,\rho_m(u)\wedge\phi_r$.

 We will show that, for each conjunct $\exists u\,\rho_m(u)\wedge\phi_r$, there is a controlled formula so that, for each $j\in J_r$, under the assumption that $\eta^+_{r,j}(\vec x,\vec y)\wedge\forall^U\vec y\,\eta^-_{r,j}(\vec x,\vec y)$, $\phi_r$ is equivalent to $\psi_{r,j}$. Then we will take $J$ to consist of functions with $f(r)\in J_r$ for each $r$, $\eta^{\pm}_f$ to be $\bigwedge_r\eta^{\pm}_{r,f(r)}$, and $\psi_f$ be $\bigwedge_r\psi_{r,f(r)}$.

  
  If $\phi_r$ is an equality or an inequality, quantifier elimination for $\RCF$ says that $\exists u\,\rho_m(u)\wedge\phi_r$ is equivalent to a quantifier-free formula, and $\exists^Uu\,\rho_m(u)\wedge\phi_r$ is a guarded formula, so Lemma \ref{thm:equiv_guarded} applies.

  If $\phi_r$ is a root formula $\exists^Uw\,t'_nw^n+t'_0=0$ then we may swap the order of the quantifiers: 
  \[\exists u\,\rho_m(u)\wedge \exists^Uw\,t'_nw^n+t'_0=0\]
  is equivalent to
  \[\exists^Uw\exists u\,\rho_m(u)\wedge t'_nw^n+t'_0=0.\]
  Again we may apply quantifier elimination for $\RCF$: this is equivalent to $\exists^Uw\,\psi'$ for some quantifier-free $\psi'$ in the language of ordered fields. Since $\exists^Uw\,\psi'$ is a guarded formula, we may apply Lemma \ref{thm:equiv_guarded} to show that this is equivalent to a permitted formula.

  If $\phi_r$ is a negated root formula $\forall^Uw\,t'_nw^n+t'_0=0$, we use the fact that that, since $\rho_m(u)$ picks out a unique element $\exists u\,\rho_m(u)\wedge\phi_r$ is equivalent to $\exists u\,\rho_m(u)\wedge\forall u\,\rho_m(u)\rightarrow\phi_r$. Then $\exists u\,\rho_m(u)$ is equivalent to a quantifier-free formula since it is in the language of ordered fields with additional constants, and we deal with $\forall u\,\rho_m(u)\rightarrow\phi_r$ by swapping the order of the quantifiers and applying Lemma \ref{thm:equiv_guarded} as we did for a root formula.

  The remaining conjunct that might appear, $\exists^U w\, \rho_m(w)$, is also a guarded formula, so we may apply Lemma \ref{thm:equiv_guarded} to this formula.
\end{proof}

\begin{lemma}\label{thm:R_exists_U_elim}
  Let $\vec x$ be variables, let $\phi(\vec x,w)$ be a permitted formula. Then $RCFU$ proves a controlled formula
  \[\forall\vec x\,(\bigoplus_{j\in J}\eta^{\pm}_j(\vec x,\vec y_j)\rightarrow\left([\exists^Uw\,(\phi(\vec x,w)]\leftrightarrow\psi_j(\vec x,\vec y_j)\right)\]
  where each $\psi_j$ is a permitted formula.
\end{lemma}
\begin{proof}
  We apply Lemma \ref{thm:R_exists_U_free_elim}, giving us a set $J_0$, formulas $\eta^\pm_j$, $\psi_j$, and a formula $\nu$.

  Therefore, assuming $\eta^+_j(\vec x,\vec y_j)\wedge\forall^U\,\eta^-_j(\vec x,\vec y_j)$, the formula $\exists^Uw\,\phi$ is equivalent to $\exists^Uw\,(\nu\vee\bigvee_{i\leq m}\left(\sum_{i'\leq n_i}t_{i,i'}w^{i'}=0\wedge\bigvee_{i'}t_{i,i'}\neq 0\right))\wedge\phi$. Distributing the disjunction and the existential, we need to show that $\exists^Uw\,\nu\wedge\phi$ and each $\exists^Uw\, \sum_{i'\leq n_i}t_{i,i'}w^{i'}=0\wedge\bigvee_{i'}t_{i,i'}\wedge\phi$ is equivalent to a permitted formula.

  The first is equivalent to the $\psi_j$ given by our application of Lemma \ref{thm:R_exists_U_free_elim}. For each other disjunct, we apply Lemma \ref{thm:R_exists_U_bound_elim}.
\end{proof}

\begin{lemma}
  Let $\vec x$ be variables and let $\phi(\vec x,w)$ be a permitted formula. Then $RCFU$ proves a controlled formula
  \[\forall\vec x\,(\bigoplus_{j\in J}\eta^{\pm}_j(\vec x,\vec y_j)\rightarrow\left((\exists w\,\neg Uw\wedge\nu(\vec x,\vec y_j,w)\wedge\phi(\vec x,w))\leftrightarrow\psi_j(\vec x,\vec y_j)\right))\]
  where $\nu$ is a conjunction of the form $\bigwedge_{i\leq n}\forall^U\vec u\, \left(\sum_{i'\leq n_{i}}t_{i,i'}w^{i'}=0\rightarrow \bigwedge_{i'\leq n_{i,i'}}t_{i,i'}\neq 0\right)$
\end{lemma}
\begin{proof}
  As usual, we may write $\phi$ in disjunctive normal form and distribute the $\exists$ over the disjunction. So we may assume $\phi$ is a conjunction of equalities, inequalities, and root formulas.

  We will take $\nu$ to be the conjunction of:
  \begin{itemize}
  \item $\sum_i t_iw^i=0\rightarrow\bigwedge_i t_i=0$ for each equality $\sum_i t_iw^i=0$ appearing in $\phi$,
  \item $\forall^Uu\, \sum_{i\leq m}(t_{n,i}u^n+t_{0,i})w^i=0\rightarrow\bigwedge_i t_i=0$ for each root formula appearing in $\phi$.
  \end{itemize}
  Under the assumption $\nu$, any equality is equivalent to $\bigwedge_i t_i=0$ and any root formula is equivalent to $\exists^Uu\,\bigwedge_{i\leq m}t_{n,i}u^n+t_{0,i}=0$.  The former is a quantifier-free formula not involving $w$. The latter is a guarded formula not involving $w$, so we may apply Lemma \ref{thm:equiv_guarded} to find a controlled formula where, in each clause and under the additional assumption $\nu$, $\phi$ is equivalent to a permitted formula in which $w$ only appears in inequalities.

  We may then remove anything which does not depend on $w$ from the existential. The remaining formulas are inequalities. Since $\RCFU$ proves that the complement of $U$ is dense, there is a $w$ not in $U$ satisfying these inequalities just if any element satisfies these inequalities, so quantifier elimination for real closed fields applies.
\end{proof}

\begin{lemma}
  Let $\vec x$ be variables, let $\phi(\vec x,w)$ be a permitted formula. Then $RCFU$ proves a controlled formula
  \[\forall\vec x\,(\bigoplus_{j\in J}\eta^{\pm}_j(\vec x,\vec y_j)\rightarrow\left([\exists w\,(\phi(\vec x,w)]\leftrightarrow\psi_j(\vec x,\vec y_j)\right)\]
  where each $\psi_j$ is a permitted formula.
\end{lemma}
\begin{proof}
  $\exists w\,\phi(\vec x,w)$ is equivalent to $\exists^Uw\,\phi(\vec x,w)\vee\exists w\,\neg Uw\wedge\phi(\vec x,w)$. Lemma \ref{thm:R_exists_U_elim} shows that the first disjunct is equivalent to a permitted formula, so it remains to consider the second.

  Applying the previous lemma to the second disjunct, we get a $\nu$ so that $\exists w\, \neg Uw\wedge\nu\wedge\phi(\vec x,w)$ is equivalent (under suitable assumptions $\eta^+$ and $\eta^-$) to a permitted formula, so we consider the case where $\nu$ fails. $\neg\nu$ is a disjunction, so we may distribute over the disjunction, so it suffices to consider some $\exists w\, \neg Uw\wedge\exists^U u\,\sum_{i\leq n}t_iw^i=0\wedge\bigvee_{i\leq n}t_i\neq 0\wedge\phi(\vec x,w)$ (where the $\exists^U$ may be missing).

  We may swap the order of the quantifiers and replace $\neg Uw$ with the negated root formula $\forall^Uv\,v\neq w$. That is, our formula is equivalent to
  \[\exists^U u\, \exists w\, \sum_{i\leq n}t_iw^i=0\wedge\bigvee_{i\leq n}t_i\neq 0\wedge (\forall^Uv\,v\neq w)\wedge \phi(\vec x,w).\]
  The matrix is a permitted formula, so we may apply Lemma \ref{thm:R_exists_U_bound_elim} followed, if necessary, by Lemma \ref{thm:R_exists_U_elim} to show that this is (under suitable assumptions $\eta^+$ and $\eta^-$) equivalent to a permitted formula.
\end{proof}

\begin{lemma}
  For any formula $\phi(\vec x)$, $\RCFU$ proves a controlled formula
  \[\forall\vec x\,(\bigoplus_{j\in J}\eta^{\pm}_j(\vec x,\vec y_j)\rightarrow\left(\phi(\vec x)\leftrightarrow\psi_j(\vec x,\vec y_j)\right)\]
  where each $\psi_j$ is a permitted formula.
\end{lemma}
\begin{proof}
  By induction on $\phi$. If $\phi$ is an equality or inequality, this is immediate. If $\phi$ is $Ut$ or $\neg Ut$, we replace it with the permitted formula $\exists^Uu\,t=u$ or $\neg\exists^Uu\,t=u$, respectively.

  The inductive cases for connectives are immediate as usual. The quantifier case is covered by the previous lemma.
\end{proof}

\begin{theorem}
  For any formula $\phi(\vec x)$, $\RCFU$ proves that $\phi$ is a equivalent to a Boolean comination of KG-formulas.
\end{theorem}
\begin{proof}
  By the previous lemma, we have
  \[\forall\vec x(\bigoplus_{j\in J}\eta^{\pm}_j(\vec x,\vec y_j)\rightarrow\left(\phi(\vec x)\leftrightarrow\psi_j(\vec x,\vec y_j)\right).\]

  Therefore $\phi$ is equivalent to
  \[\bigvee_j\forall^U\vec y_j\,\eta^-_j(\vec x,\vec y_j)\wedge\exists^U\vec y_j\,\eta^+_j(\vec x,\vec y_j)\wedge\psi_j(\vec x,\vec y_j).\]

  This is almost a Boolean combination of KG-formulas, except that we might have root formulas and negated root formulas in the $\psi_j$. As usual, we may put $\psi_j$ in disjunctive normal form and distribute the existential, so we may assume $\psi_j$ is a conjunction. For root formulas, we can pull the existential quantifiers out---that is, $\exists^U\vec y_j\,\eta^+(\vec x,\vec y_j)\wedge\exists^Uu\,su^m=t\wedge\psi'_j$ is equivalent to $\exists^U\vec y_j,u\,\eta^+(\vec x,\vec y_j)\wedge su^m=t\wedge\psi'_j$.

  We are left with negated root formulas, $\forall^Uu\,su^m\neq t$ where $s,t$ are monomials. We separate them into a monomial in $\vec x$ and a monomial in $\vec y_j$: $s=s_xs_y$ and $t=t_xt_y$. One reason we could have $\forall^Uu\,su^m\neq t$ is because $t_x/s_x$ is not in $U$, and therefore there are no choices of $s_y,t_y,u$ in $U$ with $t_x/s_x=s_yu^m/t_y$. On the other hand if $t_x/s_x$ is in $U$, we can name it as part of $\vec y_j$ and then quantifier over it.

  Formally, $\forall^Uu\,su^m\neq t$ is equivalent to $(\forall^Uv\,s_xv\neq t_x)\vee(\exists^Uv\,s_xv=t_x\wedge\forall^Uu\,s_yu^m\neq vt_y)$.  Therefore $\exists^U\vec y_j\,\eta^+_j(\vec x,\vec y_j)\wedge(\forall^Uu\,su^m\neq t)\wedge\psi'_j(\vec x,\vec y_j)$ is equivalent to
  \[\left[\forall^Uv\,s_xv\neq t_x\right]\vee\left[\exists^U\vec y_j,v\,\eta^+_j(\vec x,\vec y_j)\wedge s_xv=t_x\wedge (\forall^Uu\,s_yu^m=vt_y)\wedge \psi'_j\right].\]
  Applying this to each negated root formula in each $\psi_j$, we are left with a Boolean combination of KG-formulas.
\end{proof}

\printbibliography

\end{document}